\documentclass{amsart}

\newcommand{\apref}[3]{\hyperref[#2]{#1\ref*{#2}#3}}

\usepackage{enumerate}
\usepackage[latin1]{inputenc}
\usepackage{dsfont}
\usepackage{amssymb,amsthm,amsmath}

\input{xy}
\xyoption{all}

\usepackage{mathrsfs}

\theoremstyle{plain}
\newtheorem{prop}{Proposition}[section]
\newtheorem{lemma}[prop]{Lemma}

\newtheorem{thm}[prop]{Theorem}

\theoremstyle{definition}

\newtheorem{conj}[prop]{Conjecture}

\theoremstyle{remark}
\newtheorem{remark}[prop]{Remark}

\DeclareMathOperator{\SL}{SL}
\DeclareMathOperator{\PSL}{PSL}
\DeclareMathOperator{\PGL}{PGL}

\DeclareMathOperator{\Tr}{Tr}

\DeclareMathOperator{\Ima}{Im}
\DeclareMathOperator{\Rea}{Re}

\DeclareMathOperator{\pr}{pr}

\DeclareMathOperator{\Fix}{Fix}

\DeclareMathOperator{\CS}{CS}

\DeclareMathOperator{\NIC}{NIC}

\DeclareMathOperator{\Per}{Per}

\newcommand{\hg}{\overline \h^g}

\newcommand{\st}{\text{st}}

\newcommand{\dec}{\text{dec}}
\newcommand{\parab}{\text{par}}
\newcommand{\dyn}{\text{dyn}}

\newcommand\N{\mathbb{N}}
\newcommand\Q{\mathbb{Q}}
\newcommand\R{\mathbb{R}}
\newcommand\Z{\mathbb{Z}}
\newcommand\C{\mathbb{C}}

\newcommand{\h}{\mathbb{H}}

\newcommand{\mc}[1]{\mathcal #1}

\newcommand{\wt}{\widetilde}
\newcommand{\wh}{\widehat}

\newcommand{\eps}{\varepsilon}

\DeclareMathOperator{\FE}{FE}

\DeclareMathOperator{\id}{id}

\DeclareMathOperator{\Fct}{Fct}

\newcommand{\sceq}{\mathrel{\mathop:}=}

\newcommand{\mat}[4]{\begin{pmatrix} #1&#2\\#3&#4\end{pmatrix}}
\newcommand{\textmat}[4]{\left(\begin{smallmatrix} #1&#2 \\ #3&#4
\end{smallmatrix}\right)}

\newcommand\ie{\mbox{i.\,e., }}

\newcommand\wrt{\mbox{w.\,r.\,t.\@ }}

\usepackage[colorlinks]{hyperref}
\usepackage{graphics}
\usepackage[T1]{fontenc} 

\begin{document}

\title[Hecke triangle groups]{Period functions for Hecke triangle groups, and the Selberg zeta function as a Fredholm determinant
}
\author{M.\@ M\"oller}
\address{Goethe-Universit\"at Frankfurt, Institut f\"ur Mathematik, D-60325 Frankfurt (Main)}
\email{moeller@math.uni-frankfurt.de}
\author{A.D.\@ Pohl}
\address{ETH Z\"urich, Departement Mathematik, R\"amistrasse 101, CH-8092 Z\"urich}
\email{anke.pohl@math.ethz.ch}
\begin{abstract} 
We characterize Maass cusp forms for any cofinite Hecke triangle group as $1$-eigenfunctions of appropriate regularity of a transfer operator family. This transfer operator family is associated to a certain symbolic dynamics for the geodesic flow on the orbifold arising as the orbit space of the action of the Hecke triangle group on the hyperbolic plane. Moreover we show that the Selberg zeta function is the Fredholm determinant of the transfer operator family associated to an acceleration of this symbolic dynamics. 
\end{abstract}
\maketitle

\section{Introduction and Statement of Results}

Let $\Gamma$ be a cofinite Hecke triangle group and consider its action on two-dimensional real hyperbolic space $\h$ by Moebius transformations. It is well-known \cite{Hejhal_stf2, Selberg, Fischer} that the  zeros with $\Rea s = \tfrac12$ of the Selberg zeta function
\[
 Z(s) = \prod_{\wh\gamma} \prod_{k=0}^\infty \left(1-e^{-(s+k)l(\wh\gamma)} \right)
\]
correspond to the eigenvalues of the Maass cusp forms for $\Gamma$. Here the outer product is over all closed geodesics $\wh\gamma$ on the orbifold $\Gamma\backslash\h$ and $l(\wh\gamma)$ denotes the length of $\wh\gamma$. In case of the modular group $\PSL(2,\Z)$ an even closer relation is known. Series \cite{Series} provided a symbolic dynamics for the geodesic flow on $\PSL(2,\Z)\backslash \h$ which relates this flow to the Gauss map
\[
 K\colon\left\{
\begin{array}{ccl}
[0,1]\setminus\Q & \to & [0,1]\setminus\Q
\\
x & \mapsto & \left\{\tfrac{1}{x}\right\} = \tfrac{1}{x} - \left\lfloor\tfrac{1}{x}\right\rfloor
\end{array}
\right.
\]
in a way that closed geodesics on $\PSL(2,\Z)\backslash \h$ correspond to finite orbits of $K$. Mayer \cite{Mayer_zeta, Mayer_thermo, Mayer_thermoPSL} investigated the transfer operator with parameter $s\in\C$ associated to $K$. This is the evolution operator of densities on the unit interval $[0,1]$ which are absolutely continuous w.r.t.\@ Lebesgue-$s$-density under the Gauss map $K$, hence
\[
 \big(\mc L_{K,s} f\big)(x) = \int_{[0,1]} \delta(x-K(y))f(y)|dy|^s = \sum_{n\in\N} (x+n)^{-2s} f\left(\tfrac{1}{x+n}\right).
\]
Here $\delta$ denotes the Kronecker delta function. Mayer showed the existence of a Banach space $\mc B$ such that for $\Rea s > \tfrac12$, the transfer operator $\mc L_{K,s}$ acts on $\mc B$ and as such is nuclear of order $0$. Moreover, he proved that the transfer operator family has a meromorphic extension $\wt{\mc L}_{K,s}$ to the whole $s$-plane with values in nuclear operators of order $0$ on $\mc B$. Its possible poles are located at $s=(1-k)/2$, $k\in\N_0$. Finally, he showed that the Selberg zeta function is represented by the product of the Fredholm determinants of $\pm\wt{\mc L}_{K,s}$:
\[
 Z(s) = \det(1-\wt{\mc L}_{K,s})\cdot \det(1+ \wt{\mc L}_{K,s}).
\]
In turn, the zeros of $Z$ are determined by the $\pm 1$-eigenfunctions of $\wt{\mc L}_{K,s}$ in $\mc B$.  Efrat \cite{Efrat_spectral} proved that the values for $s$ such that there exist a $\pm 1$-eigenfunction of $\wt{\mc L}_{K,s}$ correspond to the eigenvalues of even resp.\@ odd Maass cusp forms. Extending this result to the level of the eigenfunctions, Chang and Mayer \cite{Chang_Mayer_transop} as well as Lewis and Zagier \cite{Lewis_Zagier} showed that even resp.\@ odd Maass cusp forms for $\PSL(2,\Z)$ are in linear isomorphism to the $\pm 1$-eigenspaces of Mayer's transfer operator (then automatically $\Rea s =\tfrac12$). Their proof does not make use of the Selberg trace formula. It proceeds in two steps. In a first step they show that Maass cusp forms are isomorphic to the period functions of $\PSL(2,\Z)$ (solutions of a certain $3$-term functional equation of appropriate regularity; see \cite{Bruggeman} for an alternative proof, and \cite{Lewis} for even Maass cusp forms). In a second step they prove that period functions and the $\pm 1$-eigenfunctions of $\wt{\mc L}_{K,s}$ are in linear bijection. 

Each eigendensity $f$ of $\mc L_{K,s}$ determines a $K$-invariant measure by $f|d\lambda|^s$ (here $\lambda$ denotes Lebesgue measure). In turn, certain $K$-invariant measures on $[0,1]$ determine the zeros of $Z$ and these $K$-invariant measures with $\Rea s = \tfrac12$ are in bijection to the Maass cusp forms for the modular group. More precisely, in this last step one has to take limits in $s$ of $K$-invariant measures or densities.

In this article we set out to establish these relations between Maass cusp forms, invariant measures and the Selberg zeta functions for the case that $\Gamma$ is any cofinite Hecke triangle group. With the exception of three groups (namely the Hecke triangle groups for $q\in\{3,4,6\}$, cf.\@ Section~\ref{prelims}), these Hecke triangle groups are non-arithmetic. However, they enjoy a certain symmetry which implies that the Weyl law holds for its odd Maass cusp forms \cite{Hejhal_eigenvalueshecke}. Thus odd Maass cusp forms are known to exist. On the other hand, for the non-arithmetic Hecke triangle groups, the Phillips-Sarnak conjecture \cite{DIPS,Phillips_Sarnak_weyl,Phillips_Sarnak_cuspforms} states that no even Maass cusp forms exist \cite{Hejhal_eigenvalueshecke}. 

This work is divided into two parts. In the first part, Section~\ref{part1}, we characterize Maass cusp forms for any Hecke triangle group $\Gamma$ as solutions of a certain finite-term functional equation. This functional equation arises as the characterizing equation for the $1$-eigenfunctions of a (finite term) transfer operator family $\mc L_{F,s}$ associated to a specific symbolic dynamics for the geodesic flow on $\Gamma\backslash \h$. As before, we can think of these solutions of the functional equation as period functions for $\Gamma$ or, alternatively, as densities of $F$-invariant measures. When specified to the modular group, the arising functional equation is exactly the one used by Lewis and Zagier. Thus, we provide here a dynamical interpretation of their functional equation.

This kind of approach is very sensitive to the symbolic dynamics used. The symbolic dynamics we employ here was developed by the second named author in \cite{Pohl_Symdyn2d} for this specific purpose. It enjoys some special features: The discretization of the geodesic flow it arises from is semi-conjugate to a discrete dynamical system $F$ on $\R^+$. The map $F$ decomposes into finitely many diffeomorphisms, each of which is given by a Moebius transformation induced by an element of $\Gamma$. This implies that the transfer operator family $\mc L_{F,s}$ (and hence the functional equation for the period functions) has only finitely many terms. Further, and most crucial, if $g_1,\ldots, g_n$ are the elements in $\Gamma$ used to define $F$, then the integral transform in \cite{Lewis_Zagier} applied to these elements gives a solution of the functional equation. More precisely, let $u$ be a Maass cusp form for $\Gamma$. Then we assign to $u$ and $t\in \R^+$ a certain closed $1$-form $\omega(u,t)$ (integration against Poisson kernel in Green's form; see Section~\ref{slowsys} for details). We define (integration along the imaginary axis) 
\[
 \psi(t) \sceq \int_0^{\infty} \omega(u,t).
\]
The special form of the elements $g_1,\ldots, g_n$ allows changing the path of integration as follows:
\[
 \psi(t) = \int_0^{\infty} \omega(u,t) = \sum_{j=1}^n \int_{g_j.0}^{g_j.\infty} \omega(u,t).
\]
Transformation properties of these integrals show immediately that $\psi$ satisfies the functional equation. Finally, the discrete dynamical system $(\R^+,F)$ has a certain involutive symmetry. Taking advantage of this symmetry, we can decompose any period function into an invariant one and an anti-invariant one. The invariant period functions are seen to be in bijection with even Maass cusp forms, the anti-invariant ones with odd Maass cusp forms. Now the Phillips-Sarnak conjecture translates to the non-existence of non-zero invariant period functions if $\Gamma$ is non-arithmetic.

In the second part, Section~\ref{part2}, we represent the Selberg zeta function as the Fredholm determinant of a transfer operator family. The transfer operator $\mc L_{F,s}$ is not nuclear because two of the elements $g_1,\ldots, g_n$ are parabolic. All the other elements are hyperbolic. To overcome this problem, we accelerate the symbolic dynamics and hence the discrete dynamical system $(\R^+,F)$ on these parabolic elements. Dynamically this is reflected in an acceleration of the discretization, which itself is a discretization. For a technical reason we conjugate this acceleration to a discrete dynamical system $(D,H)$ on a bounded domain $D$. The transfer operator family $\mc L_{H,s}$ is seen to be nuclear of order $0$ if $\Rea s > \tfrac12$.

Now another crucial feature of the discretization $(\R^+,F)$ comes into the picture: its coding sequences are unique and all closed geodesics are captured. In other words, closed $H$-orbits are in bijection with closed geodesics. The result is that the Selberg zeta function equals the Fredholm determinant of the transfer operator family $\mc L_{H,s}$:
\[
 Z(s) = \det(1-\mc L_{H,s}).
\]
The transfer operator family has a meromorphic extension $\wt{\mc L}_{H,s}$ to the whole $s$-plane with values in nuclear operators of order $0$ and possible poles at $s=(1-k)/2$, $k\in\N_0$. In turn, the Fredholm determinant of $\wt{\mc L}_{H,s}$ provides a meromorphic extension of the Selberg zeta function to $\C$ with possible poles at $s=(1-k)/2$, $k\in\N_0$, and the order of any pole is at most $4$. Conjecturally, for $\Rea s = \tfrac12$, the $1$-eigenfunctions of the  transfer operators $\mc L_{F,s}$ and $\wt{\mc L}_{H,s}$ are in bijection.

The discrete dynamical system $(D,H)$ inherits the symmetry of the system $(\R^+,F)$ which is responsible for the splitting of the period functions into odd and even ones. Taking advantage of this symmetry, the transfer operator family $\mc L_{H,s}$ decomposes into an even and an odd one, denoted  $\mc L_{H,s}^+$ resp.\@ $\mc L_{H,s}^-$. Consequently, the Selberg zeta function factors as
\[
 Z(s) = \det(1- \mc L_{H,s}^+)\det(1-\mc L_{H,s}^-).
\]
The symmetry and this factorization extends to the meromorphic extension $\wt{\mc L}_{H,s}$. Conjecturally, generalizing the established case for $\PSL(2,\Z)$, for $\Rea s = \tfrac12$, the $1$-eigenfunctions of $\wt{\mc L}_{H,s}^+$  should be in bijection with even Maass cusp forms, and the $1$-eigenfunctions of $\wt{\mc L}^-_{H,s}$ should correspond to odd Maass cusp forms. Proving this conjecture is the last step for the complete transfer operator method for Hecke triangle groups. We discuss this in more detail in Section~\ref{sec_conj}. In case of the modular group $\PSL(2,\Z)$, the transfer operators $\mc L_{H,s}^+$ and $-\mc L_{H,s}^-$ are Mayer's transfer operator (after the conjugation we used).

There are two other approaches to represent the Selberg zeta function for Hecke triangle groups as a Fredholm determinant of a transfer operator family. The first one is provided by Fried \cite{Fried_triangle}. He develops symbolic dynamics for the billiard flow on $\Gamma\backslash S\h$ where $\Gamma$ is a triangle group (hence, a subgroup of $\PGL(2,\R)$). Then he induces it to finite index subgroups via permutation representations, which includes a symbolic dynamics for the geodesic flow for Hecke triangle groups. The other one is provided by Mayer, Str\"omberg \cite{Mayer_Stroemberg} resp.\@ by Mayer, M\"uhlenbruch and Str\"omberg \cite{Mayer_Muehlenbruch_Stroemberg}. For the modular group $\PSL(2,\Z)$, Bruggeman and M\"uhlenbruch \cite{Bruggeman_Muehlenbruch} show that the $1$-eigenfunctions of the transfer operator family of this approach are in linear isomorphism to the Maass cusp forms. In \cite{Mayer_Stroemberg, Mayer_Muehlenbruch_Stroemberg}, the authors took the generating function for Rosen continued fractions as discrete dynamical system and developed a (rather involved) discretization of the geodesic flow on $\Gamma\backslash\h$. This discretization allows to link the Fredholm determinant of the arising transfer operator family to the Selberg zeta function. However, the coding sequences are not unique, whence the Fredholm determinant is not exactly the Selberg zeta function. Indeed, one has a representation of the form
\[
 Z(s) = h(s)\det(1- \mc L_s),
\]
where $h$ is a function that compensates for multiple coding. Moreover, results of this kind are shown by Pollicott \cite{Pollicott} for cocompact Fuchsian groups using the Series symbolic dynamics, and by \cite{Morita_transfer} for a wide class of cofinite Fuchsian groups (however, not containing Hecke triangle groups) using a modified Bowen-Series symbolic dynamics. Also multiple codings exist here, and hence the Selberg zeta function is not exactly the Fredholm determinant of the associated transfer operators. Finally, using representations Chang and Mayer \cite{Chang_Mayer_eigen, Chang_Mayer_extension} extended the case of the modular group to its finite index subgroups and hence represented for these groups the Selberg zeta function exactly as the Fredholm determinant of the transfer operator family induced by the representation. For these finite index subgroups Deitmar and Hilgert \cite{Deitmar_Hilgert} provided period functions for the Maass cusp forms of these groups also using representations. For the groups $\Gamma_0(N)$, the combination of \cite{Hilgert_Mayer_Movasati} and \cite{Fraczek_Mayer_Muehlenbruch} shows a close relation between eigenfunctions of certain transfer operators and Maass cusp forms.

\textit{Acknowledgments:}
The authors like to thank Roelof Bruggeman, Tobias Hartnick, Joachim Hilgert, Dieter Mayer, Martin Olbrich and Fredrik Str\"omberg for helpful comments respective for patiently answering questions about their work. Moreover, they like to thank Viviane Baladi for drawing their attention to the reference \cite{Fried_triangle}, and an anonymous referee for a thorough reading and helpful comments. The authors were partially supported by the Sonderforschungsbereich/Transregio 45 ``Periods, moduli spaces and arithmetic of algebraic varieties'', the Max-Planck-Institut f\"ur Mathematik in Bonn, and the second named author by the SNF (200021-127145).

\section{Preliminaries}\label{prelims}

\subsection{Hecke triangle groups}

Let $q\in\N$, $q\geq 3$. We consider the \textit{Hecke triangle group} $G_q$ with the presentation
\[
 \langle T,S \mid S^2 = \id = (TS)^q\rangle
\]
or, equivalently
\[
 \langle U,S \mid S^2 = \id = U^q \rangle.
\]
Let $\lambda \sceq 2\cos\tfrac{\pi}{q}$. Then $G_q$ is identified with the subgroup of $\PSL(2,\R)$ generated by
\[
 T \sceq \mat{1}{\lambda}{0}{1} \quad\text{and}\quad S\sceq \mat{0}{-1}{1}{0}
\]
or alternatively by
\[
 U \sceq TS = \mat{\lambda}{-1}{1}{0} \quad\text{and}\quad S.
\]
The group $\PSL(2,\R)$ is isomorphic to the group of orientation-preserving isometries of two-dimensional real hyperbolic space. As  model for this space we use the upper half plane
\[
 \h \sceq \{ z\in\C \mid \Ima z > 0\}
\]
endowed with the well-known hyperbolic metric. The action of $\PSL(2,\R)$ on $\h$ is given by fractional linear transformations, that is, $g=\textmat{a}{b}{c}{d} \in \PSL(2,\R)$ acts on $z\in \h$ by
\[
 g.z = \frac{az+b}{cz+d}.
\]

\subsection{Maass cusp forms}

A \textit{Maass cusp form} for the Hecke triangle group $G_q$ is a twice continuously differentiable function $u\colon \h\to \C$ which satisfies the following properties: It is an eigenfunction of the hyperbolic Laplace-Beltrami operator 
\[
\Delta = -y^2 \left( \frac{\partial^2}{\partial x^2} + \frac{\partial^2}{\partial y^2}\right)
\]
which is constant on each $G_q$-orbit, \ie for all $g\in G_q$ and all $z\in \h$ we have
\[ 
 u(g.z) = u(z).
\]
Moreover, $u$ has at most polynomial growth at infinity. Finally, for all $y>0$ it should hold that
\[
 \int_0^\lambda u(x+iy)dx = 0.
\]
The eigenvalue of a Maass cusp form is positive and factorizes as $s(1-s)$ with $0<\Rea s < 1$.

\subsection{Cohomological characterization of Maass cusp forms}

Recently, Bruggeman, Lewis and Zagier \cite{BLZ_part2} provided a characterization of the spaces of Maass cusp forms for arbitrary discrete cofinite subgroups of $\PSL(2,\R)$ as various first parabolic cohomology spaces. From their work we will use that the space of Maass cusp forms for $G_q$ with eigenvalue $s(1-s)$ is isomorphic (as a vector space) to the first parabolic cohomology space with values in the smooth and semi-analytic vectors of the principal series representation associated to the spectral parameter $s$. They showed that this isomorphism is explicitly given by a certain integral transform. In this section we expound the parts of this isomorphism which are important for the work at hand. 

We use the line model for principal series representation. For the spectral parameter $s\in\C$, the space $\mc V^{\omega^*,\infty}_s$ of smooth and semi-analytic vectors in the line model of the principal series representation associated to $s$ consists of complex-valued functions $\varphi$ on $\R$ which are both
\begin{itemize}
\item smooth with an asymptotic expansion of the form 
\begin{equation}\label{asymexp}
 \varphi(t) \sim |t|^{-2s} \sum_{n=0}^\infty c_nt^{-n} \quad\text{as $|t|\to\infty$}
\end{equation}
for some complex-valued sequence $(c_n)_{n\in\N_0}$ (which depends on $\varphi$),
\item and real-analytic on $\R\setminus E$ for a finite subset $E$ which depends on $\varphi$.
\end{itemize}

To define the action of $G_q$ on $\mc V_s^{\omega^*,\infty}$ we set
\[
 j_s(g,t)\sceq \big((ct+d)^{-2}\big)^s \quad (= |ct+d|^{-2s})
\]
for $s\in\C$, $g=\textmat{a}{b}{c}{d}\in \PSL(2,\R)$ and $t\in\R$. Moreover, for a function $f\colon V\to \R$ on some subset $V$ of $\R$ and $g\in\PSL(2,\R)$, we define 
\begin{equation}\label{tau_s}
 \tau_s(g^{-1}) f(t) \sceq j_s(g,t) f(g.t)
\end{equation}
whenever this makes sense. Then the  requirement \eqref{asymexp} is equivalent to impose that the map $\tau_s(S)\varphi$, that is
\[
 t\mapsto (t^{-2})^{s}\varphi\left(-\frac{1}{t}\right), 
\]
extends smoothly to $\R$. The action of $G_q$  on $\mc V_s^{\omega^*,\infty}$ is given by $\tau_s$ (with $V=\R$), hence
\[
 (\tau_s(g^{-1})\varphi)(t) \sceq \left((ct+d)^{-2}\right)^{s} \varphi\left(\frac{at+b}{ct+d}\right)
\]
for $g=\textmat{a}{b}{c}{d}\in G_q$, $\varphi\in\mc V_s^{\omega^*,\infty}$ and $t\in\R$. In the case that $c\not=0$ and $t=-\tfrac{d}{c}$, this definition is to be understood as a limit.

In the following we will recall the definition of the first parabolic cohomology space $H^1_\parab(G_q;\mc V^{\omega^*,\infty}_s)$ from \cite{BLZ_part2}. Parabolic cohomology spaces are certain subspaces of group cohomology. For group cohomology we use the standard model. A 1-cocycle in group cohomology of $G_q$ is then represented uniquely by a map $\alpha\colon G_q^2 \to \mc V_s^{\omega^*,\infty}$ which satisfies
\[
 \tau_s(g^{-1})\alpha(g_1,g_2) = \alpha(g_1g,g_2g) \qquad\text{and}\qquad \alpha(g_0,g_1) + \alpha(g_1,g_2) = \alpha(g_0,g_2)
\]
for all $g,g_0,g_1,g_2\in G_q$. A 1-coboundary in group cohomology is a 1-cocycle $\alpha$ in group cohomology for which there exists a map $f\colon G_q \to \mc V_s^{\omega^*,\infty}$ such that
\[
\tau_s(g^{-1})f(g_1) = f(g_1g) \qquad\text{and}\qquad  \alpha(g_0,g_1) = f(g_0) - f(g_1)
\]
for all $g,g_0,g_1\in G_q$. A \textit{parabolic 1-cocycle}  is a 1-cocycle $\alpha$ in group cohomology which, in addition, satisfies that for each parabolic element $p\in G_q$ there is a map $\psi\in \mc V_s^{\omega^*,\infty}$ such that
\[
 \alpha(p,1) = \tau_s(p^{-1})\psi -\psi.
\]
Here, $1$ denotes the unit element of $G_q$.
This property allows to normalize representatives of parabolic cocycle classes to vanish on specific elements. Parabolic 1-coboundaries coincide with the 1-coboundaries from group cohomology.

It is convenient to use the notation of restricted cocycles. Because of $G_q$-equivariance, each 1-cocycle is determined by its values on $G_q\times \{1\}$. Hence we may identify the 1-cocycle $\alpha\colon G_q^2 \to \mc V_s^{\omega^*,\infty}$  with the map $c\colon G_q \to \mc V_s^{\omega^*,\infty}$ given by
\[
 c_g \sceq c(g) \sceq \alpha(g,1)
\]
for all $g\in G_q$. Then the space of 1-cocycles of group cohomology of $G_q$ becomes
\[
 Z^1(G_q; \mc V_s^{\omega^*,\infty}) = \{ c \colon G_q \to \mc V_s^{\omega^*,\infty} \mid \forall\, g,h\in G_q \colon c_{gh} = \tau_s(h^{-1})c_g + c_h \},
\]
and that of parabolic 1-cocycles is
\[
\begin{aligned}
 Z^1_\parab(G_q; \mc V_s^{\omega^*,\infty}) = \{ c \in Z^1(G_q;\mc V_s^{\omega^*,\infty}) \mid \forall\,  & p\in G_q\ \text{parabolic}\ \exists\, \psi\in\mc V_s^{\omega^*,\infty}\colon \\
 & c_p = \tau_s(p^{-1})\psi - \psi \}.
\end{aligned}
\]
The spaces of  1-coboundaries of group cohomology and of parabolic 1-coboundaries  coincide. They are
\[
 B^1_\parab(G_q;\mc V_s^{\omega^*,\infty}) = B^1(G_q; \mc V_s^{\omega^*,\infty}) = \{ g \mapsto \tau_s(g^{-1})\psi - \psi \mid \psi \in \mc V_s^{\omega^*,\infty} \}.
\]
Finally, the \textit{parabolic 1-cohomology space} is the quotient space
\[
 H^1_\parab(G_q; \mc V_s^{\omega^*,\infty}) = Z^1_\parab(G_q; \mc V_s^{\omega^*,\infty})/ B^1_\parab(G_q;\mc V_s^{\omega^*,\infty}).
\]

\begin{thm}[\cite{BLZ_part2}]\label{BLZ_main}
The space of Maass cusp forms for $G_q$ with eigenvalue $s(1-s)$ is in linear isomorphism with $H^1_\parab(G_q;\mc V_s^{\omega^*,\infty})$.
\end{thm}

The isomorphism from Theorem~\ref{BLZ_main} is given by an integral transform which we now recall. We define the function $R\colon \R \times \h \to \R$ by
\[
 R(t,z) \sceq \Ima\left( \frac{1}{t-z} \right).
\]
For two eigenfunctions $u,v$ of the Laplace-Beltrami operator $\Delta$, Green's form shall be denoted by
\[
 [u,v] \sceq \frac{\partial u}{\partial z}\cdot v dz + u \cdot \frac{\partial v}{\partial \overline z} d\overline z.
\]
We choose any $z_0\in \h$. Then the  parabolic 1-cocycle class in $H^1_\parab(G_q; \mc V^{\omega^*,\infty}_s)$  associated to the Maass cusp form $u$ with eigenvalue $s(1-s)$ is represented by the cocycle $c$ with
\begin{equation}\label{integraltransform}
 c_g(t) \sceq \int_{g^{-1}.z_0}^{z_0} [u,R(t,\cdot)^s].
\end{equation}
If $p\in G_q$ is parabolic and $x\in\R\cup\{\infty\}$ is fixed by $p$, then
\begin{equation}\label{coboundproof}
 \psi(t) \sceq \int_{z_0}^x [u, R(t,\cdot)^s]
\end{equation}
is in $\mc V_s^{\omega^*,\infty}$ and $c_p=\tau_s(p^{-1})\psi - \psi$. Here, the integration is along a path in $\h$ from $z_0$ to $x$ where only the endpoints are allowed to be in $\R\cup\{\infty\}$. Changing the choice of $z_0$ changes $c$ by a parabolic $1$-coboundary. 

To end this section, we show that we may take $z_0=\infty$ in \eqref{integraltransform} to define a representative of the $1$-cocycle class $[c]$. So let $z_0$ be any element in $\h$ and define a representative $c\colon G_q \to \mc V_s^{\omega^*,\infty}$ of $[c]$ via \eqref{integraltransform}. The parabolic element $T\in G_q$ fixes $x\sceq \infty$. Hence
\[
 c_T=\tau_s(T^{-1})\psi -\psi
\]
with $\psi$ as in \eqref{coboundproof}. Since $d\colon g\mapsto \tau_s(g^{-1})\psi -\psi$ is a parabolic $1$-coboundary, the parabolic $1$-cocycle $\wt c \sceq c-d$ represents the $1$-cocycle class $[c]$. As in \cite[Chap.\@ II.2]{Lewis_Zagier} it follows that
\[
 \tau_s(g^{-1})\int_{a}^{b} [u,R(t,\cdot)^s] = \int_{g^{-1}.a}^{g^{-1}.b} [u,R(t,\cdot)^s]
\]
for $a,b\in \h \cup \R \cup \{\infty\}$ and $g\in G_q$. Thus
\begin{align*}
 \wt c_g(t) &= \int_{g^{-1}.z_0}^{z_0} [u,R(t,\cdot)^s] - \int_{g^{-1}.z_0}^{g^{-1}.\infty} [u,R(t,\cdot)^s]  - \int_{z_0}^{\infty} [u,R(t,\cdot)^s] 
\\
&= \int_{g^{-1}.\infty}^{\infty} [u,R(t,\cdot)^s]. 
\end{align*}
We remark that the integration is performed along a path from $g^{-1}.\infty$ to $\infty$ in $\h$.

\section{Maass cusp forms and period functions}\label{part1}

The aim of this section is to show the bijection between solutions
of a finite-term functional equation and parabolic cohomology
classes, stated as Theorem~\ref{periodthm} below. Together with the
result of \cite{BLZ_part2} this establishes the bijection between certain solutions
of the functional equation and Maass cusp forms.
The main tool is the symbolic dynamics of the geodesic flow
constructed in \cite{Pohl_Symdyn2d} which was designed so that
$1$-eigenfunctions of the associated transfer operator satisfy
the functional equation.
In Section~\ref{oddeven} we show that the splitting of the solutions of the functional equation
into an odd and an even part 
is an immediate consequence of a symmetry of the symbolic dynamics and hence of the functional equation.

\subsection{Slow discrete dynamical system}\label{slowsys}

The second named author constructed in \cite{Pohl_Symdyn2d}  a discretization (symbolic dynamics) for the geodesic flow on the orbifold $G_q\backslash \h$ which is semi-conjugate to a discrete dynamical system on  $(0,\infty)\setminus G_q.\infty$. In this section we introduce this discrete dynamical system and the transfer operator family associated to it. Using the integral transform \eqref{integraltransform} with $z_0=\infty$, the geometry of the symbolic dynamics resp.\@ of the discrete dynamical system yields  immediately that Maass cusp forms give rise to 1-eigenfunctions of the transfer operator. 

We set $D_\st \sceq (0,\infty) \setminus G_q.\infty$. Recall the elements $U$ and $S$ of $G_q$ from Section~\ref{prelims}. For $k \in \N$ we define
\[
 g_k \sceq \big(U^kS)^{-1} = \frac{1}{\sin\frac{\pi}{q}} \mat{ \sin\big(\tfrac{k}{q}\pi\big) }{ -\sin\big(\tfrac{k+1}{q}\pi\big) }{ -\sin\big(\tfrac{k-1}{q}\pi\big) }{ \sin\big(\frac{k}{q}\pi\big) }.
\]
For $k=1,\ldots, q-1$ we set
\[
 D_{\st,k} \sceq \big(g_k^{-1}.0, g_k^{-1}.\infty \big) \cap D_\st.
\]
Then the discrete dynamical system developed in \cite{Pohl_Symdyn2d} is
\begin{equation}\label{slowsystem}
F\colon D_\st \to D_\st,\quad  F\vert_{D_{\st,k}} \sceq g_k \quad\text{for $k=1,\ldots, q-1$.}
\end{equation}
We shall refer to the system $(D_\st, F)$ as the \textit{slow discrete dynamical system}. Let $\Fct(D_\st; \C)$ denote the set of complex-valued functions on $D_\st$. The \textit{transfer operator with parameter $s\in\C$ associated to $F$} is the map
\[
 \mc L_{F,s}\colon \Fct(D_\st;\C) \to \Fct(D_\st;\C)
\]
given by 
\[
 \mc L_{F,s} = \sum_{k=1}^{q-1}\tau_s(g_k).
\]
In the following we will use $\mc L_{F,s}$ as an operator on $\Fct(\R^+;\C)$, which is obviously well-defined. Suppose that $u$ is a Maass cusp form for $G_q$ with eigenvalue $s(1-s)$. In the integral formula \eqref{integraltransform} let $z_0 \sceq \infty$ and set (for $t\in\R^+$)
\[
 \varphi \sceq c_S\colon t \mapsto \int_0^\infty [u,R(t, \cdot)^s].
\]
Because $g_k.0 = g_{k+1}.\infty$ for $k=1,\ldots, q-2$, and $g_1.\infty=\infty$, $g_{q-1}.0=0$, it follows that
\begin{equation}\label{modperiod}
 \varphi(t) = \int_0^\infty [u,R(t,\cdot)^s] = \sum_{k=1}^{q-1} \int_{g_k.0}^{g_k.\infty} [u,R(t,\cdot)^s] = \sum_{k=1}^{q-1}\tau_s(g_k)\varphi(t).
\end{equation}
Hence $\varphi$ is a 1-eigenfunction of the transfer operator $\mc L_{F,s}$.

\subsection{Period functions and parabolic cohomology}

We have just seen that any Maass cusp form with eigenvalue $s(1-s)$ gives rise to a $1$-eigenfunction of the transfer operator $\mc L_{F,s}$. In this section we will determine which $1$-eigenfunctions arise in this way.
For  $s\in \C$ we let $\FE_s(\R^+)_\omega^\dec$ be the space of functions $\psi\in C^\omega(\R^+;\C)$ such that $\psi$ satisfies the functional equation
\begin{equation}\label{funceq}
 \psi = \sum_{k=1}^{q-1}\tau_s(g_k)\psi.
\end{equation}
and has asymptotic expansions of the form
\begin{equation}
\label{ae1} 
\psi(t) \sim \sum_{m=0}^\infty C^*_m t^m \hspace{2cm}\hphantom{xxxx}\text{as $t\searrow 0$,}
\end{equation}
\begin{equation}
\label{ae2}
\psi(t) \sim t^{-2s} \sum_{m=0}^\infty D_m^* t^{-m} \hspace{2cm}\text{as $t\nearrow\infty$}
\end{equation}
with
\begin{equation}\label{condae}
 C^*_m = (-1)^{m+1} D_m^* \quad\text{for all $m\in\N_0$.}
\end{equation}
We call these functions \textit{period functions} for $G_q.$

\begin{remark}\label{equivasymp}
The asymptotic expansions \eqref{ae1}-\eqref{condae} are equivalent to the requirement that 
\begin{align*}
 \psi(t) & \sim \sum_{m=0}^\infty C_m^*t^m && \hspace{-2cm}\text{as $t\searrow 0$,}
\\
-\tau_s(S)\psi(t) & \sim \sum_{m=0}^\infty C_m^*t^m &&  \hspace{-2cm}\text{as $t\nearrow 0$.}
\end{align*}
This is equivalent to that the map
\[
 t \mapsto
\begin{cases}
 \psi(t) & \text{for $t\in\R^+$,}
\\
-\tau_s(S)\psi(t) & \text{for $t\in\R^-$}
\end{cases}
\]
extends smoothly to $\R$.
\end{remark}

\begin{remark}
In the case of the modular group, that is for $q=3$, the asymptotic expansions \eqref{ae1}-\eqref{condae} are equivalent to the decay conditions (see \cite{Lewis_Zagier})
\[
 \psi(t) = o\big(\tfrac1t\big) \quad (\text{as $t\searrow 0$}),\qquad \psi(t) = o(1)\quad (\text{as $t\nearrow\infty$}).
\]
This is due to the fact that the modular group enjoys some arithmetic features which are not valid for other Hecke triangle groups.
\end{remark}

In this section we show that the first parabolic cohomology space $H^1_\parab(G_q;\mc V_s^{\omega^*,\infty})$ and the space of period functions $\FE_s(\R^+)^\dec_\omega$ are isomorphic. Given a parabolic $1$-cocycle class $[c]\in H^1_\parab(G_q;\mc V_s^{\omega^*,\infty})$ there is a unique representative $c\in Z^1_\parab(G_q;\mc V_s^{\omega^*,\infty})$ with $c_T=0$ (take \eqref{integraltransform} with $z_0=\infty$). To $[c]$ we associate the map $\psi = \psi([c]) \colon \R^+\to \C$ defined by
\[
 \psi \sceq c_S\vert_{\R^+}.
\]
Conversely, for $\psi\in \FE_s(\R^+)^\dec_\omega$ we define a map $c=c(\psi) \colon G_q \to \mc V_s^{\omega^*,\infty}$ by $c_T\sceq 0$ and 
\[
 c_S \sceq
\begin{cases}
\psi & \text{on $\R^+$,}
\\
-\tau_s(S)\psi & \text{on $\R^-$.} 
\end{cases}
\]

\begin{thm}\label{periodthm}
The map 
\[
\left\{
\begin{array}{ccl}
\FE_s(\R^+)^\dec_\omega & \to & H^1_\parab(G_q; \mc V_s^{\omega^*,\infty})
\\
\psi & \mapsto & [c(\psi)]
\end{array}
\right.
\]
defines a linear isomorphism between $\FE_s(\R^+)^\dec_\omega$ and $H^1_\parab(G_q; \mc V_s^{\omega^*,\infty})$. Its inverse is given by 
\[
\left\{
\begin{array}{ccl}
H^1_\parab(G_q; \mc V_s^{\omega^*,\infty})& \to & \FE_s(\R^+)^\dec_\omega
\\{ }
[c] & \mapsto & \psi([c]).
\end{array}
\right.
\]
\end{thm}

To prove Theorem~\ref{periodthm} we have to show that the two maps are well-defined. This is split into Propositions~\ref{cocycletoperiod}  and \ref{periodtococ} below.

\begin{prop}\label{cocycletoperiod}
Let $[c]\in H^1_\parab(G_q; \mc V_s^{\omega^*,\infty})$. Then $\psi([c])\in \FE_s(\R^+)^\dec_\omega$.
\end{prop}

\begin{proof}
Let $\psi\sceq \psi([c])$. By definition,
\[
 \psi(t) = c_S(t) = \int_{0}^{\infty} [u, R(t,\cdot)^s]
\]
for a (unique) Maass cusp form $u$ with eigenvalue $s(1-s)$. Here the integral is performed along the imaginary axis. One easily sees that this parameter integral is real-analytic on $\R^+$. Since  $c_S = -\tau_s(S)c_S$, Remark~\ref{equivasymp} yields the required asymptotic expansions for $\psi$. The discussion at the end of Section~\ref{slowsys} already shows that $\psi$ satisfies \eqref{funceq}.
\end{proof}

For Proposition~\ref{periodtococ} below we need the following lemma. The method of its proof is adapted from \cite{BLZ_part2}. 

\begin{lemma}\label{varphifun}
Suppose that  $\psi\colon \R^+\to \C$ satisfies \eqref{funceq}. Define $\varphi\colon \R\setminus\{0\} \to \C$ by
\begin{align*}
\varphi & \sceq 
\begin{cases}
\psi &\text{on $\R^+$,}
\\
 -\tau_s(S)\psi &\text{on $\R^-$.}
\end{cases}
\end{align*} 
Then $\varphi$ satisfies the functional equation \eqref{funceq} on $\R\setminus\{0\}$.
\end{lemma}

\begin{proof}
It suffices to show that $\varphi$ satisfies \eqref{funceq} on $\R^-$. Let $t\in \R^-$ . One easily checks that $g_l^{-1}.t < 0$ if and only if 
\[
 -g_l^{-1}.\infty < t < -g_l^{-1}.0.
\]
Suppose that $t\in \big( -g_l^{-1}.\infty, -g_l^{-1}.0\big)$. Then
\begin{align*}
\varphi(t) - \sum_{k=1}^{q-1}\tau_s(g_k) \varphi(t) & = \varphi(t)- \sum_{{k=1}\atop {k\not=l}}^{q-1}\tau_s(g_k)\varphi(t) - \tau_s(g_l)\varphi(t)
\\
& = -\tau_s(S)\psi(t) - \sum_{{k=1} \atop {k\not=l}}^{q-1}\tau_s(g_k)\psi(t) + \tau_s(g_l) \tau_s(S) \psi(t)
\\
& = \tau_s(g_lS) \left[ -\tau_s(Sg_l^{-1}S)\psi - \sum_{{k=1} \atop {k\not=l}}^{q-1}\tau_s(Sg_l^{-1}g_k)\psi + \psi \right](t)
\\
&= \tau_s(g_lS) \left[ -\tau_s\big( (U^{-l}S)^{-1}\big)\psi -\sum_{{k=1} \atop {k\not=l}}^{q-1} \tau_s\big( (U^{k-l}S)^{-1} \big)\psi + \psi  \right](t)
\\
& = \tau_s(g_lS) \left[ -\sum_{{k=0} \atop {k\not=l}}^{q-1}\tau_s\big( (U^{k-l}S)^{-1} \big)\psi + \psi\right](t)
\\
& = \tau_s(g_lS) \left[-\sum_{k=1}^{q-1} \tau_s(g_k)\psi + \psi\right](t)
\end{align*}
which vanishes because $\psi$ satisfies \eqref{funceq}. Therefore \eqref{funceq} is satisfied by $\varphi$ on $\R^-$.
\end{proof}

\begin{prop}\label{periodtococ}
Let $\psi\in\FE_s(\R^+)^\dec_\omega$. Then $c(\psi)$ is in $Z^1_\parab(G_q;\mc V_s^{\omega^*,\infty})$.
\end{prop}

\begin{proof}
Let $c\sceq c(\psi)$. Obviously, $c_S$ is real-analytic on $\R$ up to a finite subset. By Remark~\ref{equivasymp}, $c_S$ and $\tau_s(S)c_S$ extend smoothly to $\R$. Thus, $c_S$ determines an element in $\mc V_s^{\omega^*,\infty}$. Hence it only remains to show that $c$ is well-defined. For this it suffices to prove that $c_{S^2}$ and $c_{U^q}$ vanish. At first we see that
\[
 c_{S^2}(t) = \tau_s(S) c_S(t) + c_S(t) = 
\begin{cases}
-\psi(t) + \psi(t) = 0 & \text{for $t>0$,}
\\
\tau_s(S)\psi(t) - \tau_s(S)\psi(t) = 0& \text{for $t<0$.} 
\end{cases}
\]
Since $c_{S^2}$ is smooth on $\R$, it vanishes for $t=0$ as well.
Lemma~\ref{varphifun} and regularity yield
\[
 c_S = \sum_{k=1}^{q-1} \tau_s(g_k)c_S
\]
on $\R$. Thus we find
\begin{align*}
c_{U^q} & = \sum_{k=0}^{q-1} \tau_s(U^{-k})c_S = c_S + \sum_{k=1}^{q-1} \tau_s(U^{-k})c_S
\\
& = c_S + \tau_s(S) \sum_{k=1}^{q-1} \tau_s( SU^{-k}) c_S 
\\
& = c_S + \tau_s(S) \sum_{k=1}^{q-1} \tau_s(g_k) c_S 
\\
& = c_S + \tau_s(S) c_S = 0.
\end{align*}
This completes the proof.
\end{proof}

\subsection{Odd and even Maass cusp forms} \label{oddeven}
A Maass cusp form $u$ is called \textit{even} if $u(z) = u(-\overline{z})$ for all $z\in \h$. It is called \textit{odd} if $u(z) = -u(-\overline{z})$ for all $z\in \h$. The space of Maass cusp forms decomposes as a direct sum into the spaces of even Maass cusp forms and odd ones. Let 
\[
 Q\sceq \mat{0}{1}{1}{0} \qquad \in \PGL(2,\R).
\]
Then \cite[Equation~(2.4)]{Lewis_Zagier} (which holds \textit{verbatim} for Hecke triangle groups) shows that the even Maass cusp forms with eigenvalue $s(1-s)$ correspond to the period functions $\psi \in \FE_s(\R^+)^\dec_\omega$ which satisfy $\psi = \tau_s(Q)\psi$, and the odd Maass cusp forms with eigenvalue $s(1-s)$ correspond to the period functions $\psi\in\FE_s(\R^+)^\dec_\omega$ with $\psi= -\tau_s(Q)\psi$. Let
\begin{align*}
\FE_s(\R^+)^{\dec,+}_\omega & \sceq \{ \psi\in\FE_s(\R^+)^\dec_\omega \mid \tau_s(Q)\psi = \psi \}
\intertext{resp.\@}
\FE_s(\R^+)^{\dec,-}_\omega & \sceq \{ \psi\in\FE_s(\R^+)^\dec_\omega \mid -\tau_s(Q)\psi = \psi \}
\end{align*}
denote the set of \textit{even} resp.\@ \textit{odd period functions}. 
In this section we will show that the functional equation \eqref{funceq} and the invariance resp.\@ anti-invariance under $\tau_s(Q)$ can be encoded in a single functional equation in each case.

For this we let $m\sceq \lfloor \tfrac{q+1}{2}\rfloor$. For $q$ even, we consider the functional equations
\begin{align}
\label{funceqmod1}
\psi & = \sum_{k=m+1}^{q-1} \tau_s(g_k)\psi + \tfrac12\tau_s(g_m)\psi + \tfrac12\tau_s(Qg_m)\psi + \sum_{k=m+1}^{q-1} \tau_s(Qg_k)\psi,
\\
\label{funceqmod2}
\psi & = \sum_{k=m+1}^{q-1} \tau_s(g_k)\psi + \tfrac12\tau_s(g_m)\psi - \tfrac12\tau_s(Qg_m)\psi - \sum_{k=m+1}^{q-1}\tau_s(Qg_k)\psi.
\end{align}
For $q$ odd, we consider the functional equations
\begin{align}
\label{funceqmod3}
\psi & = \sum_{k=m}^{q-1} \tau_s(g_k)\psi + \sum_{k=m}^{q-1} \tau_s(Qg_k)\psi,
\\
\label{funceqmod4}
\psi & = \sum_{k=m}^{q-1} \tau_s(g_k)\psi - \sum_{k=m}^{q-1} \tau_s(Qg_k)\psi.
\end{align}

Invoking the crucial identity
\[
Qg_k = g_{q-k}Q \qquad\text{for all $k\in\Z$},
\]
which follows from $QS=SQ$ and $QU=U^{-1}Q$, the proof of the following proposition is straightforward.

\begin{prop}
If $q$ is even, then
\begin{align*}
\FE_s(\R^+)^{\dec,+}_\omega &= \{ \psi\in C^\omega(\R^+;\C) \mid \text{$\psi$ satisfies \eqref{ae1}-\eqref{condae} and \eqref{funceqmod1}}\}
\intertext{and}
\FE_s(\R^+)^{\dec,-}_\omega &= \{ \psi\in C^\omega(\R^+;\C) \mid \text{$\psi$ satisfies \eqref{ae1}-\eqref{condae} and \eqref{funceqmod2}}\}.
\end{align*}
If $q$ is odd, then
\begin{align*}
\FE_s(\R^+)^{\dec,+}_\omega &= \{ \psi\in C^\omega(\R^+;\C) \mid \text{$\psi$ satisfies \eqref{ae1}-\eqref{condae} and \eqref{funceqmod3}}\}
\intertext{and}
\FE_s(\R^+)^{\dec,-}_\omega &= \{ \psi\in C^\omega(\R^+;\C) \mid \text{$\psi$ satisfies \eqref{ae1}-\eqref{condae} and \eqref{funceqmod4}}\}.
\end{align*}
\end{prop}

\section{The Selberg zeta function as a Fredholm determinant}\label{part2}

In this section we show that the Selberg zeta function for $G_q$ equals the Fredholm determinant of a transfer operator family arising from a discretization for the geodesic flow on $G_q\backslash \h$, which is closely related to the slow discrete dynamical system $(D_\st, F)$ from Section~\ref{slowsys}. Since the elements 
\[
 g_1= \mat{1}{-\lambda}{0}{1} \quad\text{and}\quad g_{q-1} = \mat{1}{0}{-\lambda}{1}
\]
in $(D_\st, F)$ are parabolic (the elements $g_2,\ldots, g_{q-2}$ are hyperbolic), the associated transfer operator $\mc L_{F,s}$ is not nuclear and hence does not have a Fredholm determinant. 

To overcome this problem, we accelerate $(D_\st,F)$ on parabolic elements. The outcome will be a discretization for the geodesic flow of which the associated transfer operator family $\mc L_{H,s}$ is seen to be nuclear of order $0$ on an appropriate Banach space. For its Fredholm determinant we find
\[
 Z(s) = \det(1-\mc L_{H,s}).
\]
Throughout we use the following notation: $F$ is the slow system from Section~\ref{slowsys}, $G$ its acceleration
and $H$ is conjugate to $G$ by a certain Moebius transformation to be defined in \eqref{mcT}.
\subsection{Fast discrete dynamical system}

In the following we construct the acceleration of the slow discrete dynamical system $(D_\st,F)$ on its parabolic elements. A crucial feature of the accelerated discrete dynamical system is that it arises from a cross section for the geodesic flow on $G_q\backslash\h$. This in turn is inherited from the fact that already $(D_\st, F)$ arose from a cross section (see below).

Let $Y\sceq G_q \backslash \h$. Suppose that $\mu$ is a measure on the set of geodesics on $Y$. A subset $\wh\CS$ of the unit tangent bundle of $Y$ is called a \textit{cross section} (w.r.t.\@ $\mu$) for the geodesic flow on $Y$ if $\mu$-almost every geodesic on $Y$ intersects $\wh\CS$ infinitely often in past and future and if each such intersection is discrete in time. The first requirement means that for $\mu$-almost every geodesic $\wh\gamma$ on $Y$ there is a bi-infinite sequence $(t_n)_{n\in\Z}$ such that 
\[
\lim_{n\to \pm \infty} t_n = \pm \infty
\]
and $\wh\gamma'(t_n) \in \wh\CS$ for all $n\in\Z$. The latter means that for each intersection time $t$, that is $\wh\gamma'(t)\in \wh\CS$, there exists $\eps >0$ such that $\wh\gamma'( (t-\eps,t+\eps) ) \cap \wh\CS = \{ \wh\gamma'(t) \}$. Let $\pi\colon S\h \to SY$ denote the canonical quotient map from the unit tangent bundle of $\h$ to that of $Y$. A \textit{set of representatives} for the cross section $\wh\CS$ (w.r.t.\@ $\mu$) is a subset $\CS'$ of $S\h$ such that $\pi$ induces a bijection $\CS'\to \wh\CS$. For each $v\in S\h$ let $\gamma_v$ be the (unit speed) geodesic on $\h$ determined by $\gamma_v'(0)=v$. 

Let $\mc N$ denote the set of all unit tangent vectors $v\in S\h$ such that $\gamma_v(\infty) \in G_q.\infty$ or $\gamma_v(-\infty) \in G_q.\infty$. Further let $\NIC$ (`not infinitely often coded') denote the set of geodesics on $Y$ with unit tangent vectors in $\pi(\mc N)$, and suppose that $\mu$ is a measure on the set of geodesics on $Y$ such that $\NIC$ is a $\mu$-null set. The cross section from which the slow discrete dynamical system $(D_\st,F)$ arises is given as follows \cite{Pohl_Symdyn2d}: A  set of representatives is 
\[
 \CS'_F \sceq \left\{ a\frac{\partial}{\partial x}\vert_{iy} + b\frac{\partial}{\partial y}\vert_{iy} \in S\h \left\vert\ y>0,\ a>0,\ b\in\R \vphantom{\frac{\partial}{\partial x}\vert_{iy}}\right.\right\}\setminus\mc N.
\]
This is the set of unit tangent vectors in $S\h$ which are based on $i\R^+$ and point into the right half space $\{z\in\C \mid \Rea z>0\}$ such that the determined geodesics do not end or start in a cuspidal point. The cross section is then 
\[
 \wh\CS_F \sceq \pi(\CS'_F).
\]
It is a cross section w.r.t.\@ $\mu$, hence in particular all closed geodesics intersect $\wh\CS_F$ infinitely often. 

Let $R_F$ be the \textit{first return map} w.r.t.\@ the cross section $\wh\CS_F$, that is the map 
\[
 R_F\colon\left\{
\begin{array}{ccl}
 \wh\CS_F & \to & \wh\CS_F
\\
\wh v & \mapsto & \wh{\gamma_v}'(t_0)
\end{array}
\right.
\]
where $v\sceq \big(\pi\vert_{\CS'_F})^{-1}(\wh v)$, $\wh{\gamma_v}\sceq \pi(\gamma_v)$ and
\[
 t_0 \sceq \min\{ t>0 \mid \wh{\gamma_v}'(t) \in \wh\CS_F \}.
\]
In \cite{Pohl_Symdyn2d} it is shown that this minimum exists. Further define $\tau\colon \wh\CS_F \to \R^2$ by $\tau(\wh v) \sceq \big(\gamma_v(\infty),\gamma_v(-\infty)\big)$ and let $\pr_1\colon \R^2 \to \R$ be the projection on the first component. 
Then $D_\st$ is the image of $\wh\CS_F$ under $\pr_1 \circ \tau$ and $F$ is the unique self-map
of $D_\st$ such that the diagram
\[
\xymatrix{
\wh\CS_F \ar[r]^{R_F} \ar[d]_{\tau} & \wh\CS_F\ar[d]^{\tau}
\\
\R^2 \ar[d]_{\pr_1} & \R^2 \ar[d]^{\pr_1}
\\
D_\st \ar[r]^F & D_\st
}
\]
commutes and $\pr_1\circ\tau$ is surjective (for details see \cite{Pohl_Symdyn2d}). To derive from this the accelerated system we set
\begin{align*}
\wh{\mc N}_G & \sceq \big\{ \wh v\in \wh\CS_F\ \big\vert\  \pr_1(\tau(R_F^{-1}(\wh v))), \pr_1(\tau(\wh v)) \in \big( g_1^{-1}.0, g_1^{-1}.\infty\big)\big\}
\\
& \qquad \cup \big\{ \wh v\in \wh\CS_F \ \big\vert\  \pr_1(\tau(R_F^{-1}(\wh v))), \pr_1(\tau(\wh v)) \in \big(g_{q-1}^{-1}.0, g_{q-1}^{-1}.\infty\big) \big\}
\\
& = \big\{ \wh v\in\wh\CS_F\ \big\vert\ \pr_1(\tau(R_F^{-1}(\wh v))) \in \big(g_1^{-2}.0, g_1^{-1}.\infty\big) \big\}
\\
& \qquad \cup \big\{ \wh v\in\wh\CS_F\ \big\vert\ \pr_1(\tau(R_F^{-1}(\wh v))) \in \big(g_{q-1}^{-1}.0, g_{q-1}^{-2}.\infty\big) \big\}.
\end{align*}
These are $R_F$-images of the elements $\wh w$ in $\wh\CS_F$  for which $R_F(\wh w)$ and $R^2_F(\wh w)$ project to points in $D_{\st,1}$ resp.\@ $D_{\st,q-1}$. Let $\mc N_G\sceq \pi^{-1}(\wh{\mc N}_G)$ and define
\[
 \CS'_G \sceq \CS'_F\setminus\mc N_G \quad\text{and}\quad \wh\CS_G \sceq \wh\CS_F\setminus \wh{\mc N}_G = \pi\big(\CS'_G\big).
\]

One easily checks the following lemma.

\begin{lemma}\label{fastcs}\mbox{ }
\begin{enumerate}[{\rm (i)}]
\item\label{fastcsii} Let $\mu$ be a measure on the set of geodesics on $Y$ such that $\NIC$ is a $\mu$-null set. Then $\wh\CS_G$ is a cross section \wrt $\mu$ for the geodesic flow on $Y$.
\item\label{fastcsi} $\CS'_G$ is a set of representatives for $\wh\CS_G$.
\item Each closed geodesic on $Y$ intersects $\wh\CS_G$ infinitely often.
\end{enumerate}
\end{lemma}

Let $R_G$ be the first return map \wrt $\wh\CS_G$ and let $\tau$ be defined analogous as before. As above, there is a unique discrete dynamical system $(D_\st, G)$ such that $\pr_1\circ\tau$ is surjective and the diagram
\[
 \xymatrix{
\wh\CS_G \ar[r]^{R_G} \ar[d]_{\tau} & \wh\CS_G\ar[d]^{\tau}
\\
\R^2 \ar[d]_{\pr_1} & \R^2 \ar[d]^{\pr_1}
\\
D_\st \ar[r]^G & D_\st
}
\]
commutes. To give an explicit formula for the map $G$ we set for $n\in\N$,
\begin{align*}
D_{\st, 1, n} & \sceq \big(g_1^{-n}.0, g_1^{-(n+1)}.0\big) \cap D_\st  = ( n\lambda, (n+1)\lambda) \cap D_\st,
\intertext{and}
D_{\st, q-1, n} & \sceq \big(g_{q-1}^{-(n+1)}.\infty, g_{q-1}^{-n}.\infty\big)  \cap D_\st  = \left( \frac{1}{(n+1)\lambda}, \frac{1}{n\lambda} \right) \cap D_\st.
\end{align*}
Then the map $G\colon D_\st \to D_\st$ is given by the diffeomorphisms
\[
 G\vert_{D_{\st,k}} \sceq g_k \colon D_{\st,k} \to D_\st \qquad\text{for $k=2,\ldots, q-2$,}
\]
and
\begin{align*}
G\vert_{D_{\st, 1,n}} & \sceq g_1^n \colon D_{\st,1,n} \to D_{\st}\setminus D_{\st,1} 
\\
G\vert_{D_{\st, q-1,n}} & \sceq g_{q-1}^n \colon D_{\st,q-1,n} \to D_\st \setminus D_{\st,q-1}.
\end{align*}
for $n\in\N$.  We call $(D_\st,G)$ the \textit{fast discrete dynamical system}.

The transfer operator with parameter $s$ associated to $G$ is
\[
 \mc L_{G,s} = \sum_{n\in\N} \chi_{D_\st\setminus D_{\st,1}} \cdot \tau_s(g_1^n) + \sum_{n\in\N} \chi_{D_\st\setminus D_{\st,q-1}}\cdot \tau_s(g_{q-1}^n) + \sum_{k=2}^{q-2} \tau_s(g_k),
\]
acting on a space of functions still be to defined. Our goal is to find a domain of definition on which $\mc L_{G,s}$ becomes a nuclear operator of order $0$ and which contains real-analytic functions on $\R^+$ of certain decay. Here we encounter the problem that we have to work with a neighborhood of $[0,\infty]$ in the geodesic boundary $P^1_\R$ of $\h$.  To avoid changes of charts, we conjugate the dynamical system $(D_\st, G)$  and its associated transfer operators with the transformation
\begin{equation}\label{mcT}
 \mc T \sceq \tfrac{1}{\sqrt{2}}\mat{1}{-1}{1}{1}\colon t\mapsto \frac{t-1}{t+1}.
\end{equation}
The transformation $\mc T$ has the following geometric interpretation, which immediately reveals all its important properties. 
Let $\hg$ denote the geodesic closure of $\h$, and let $\mathbb{D} = \{ z \in \C \mid |z| < 1\}$ be the unit disc model of the hyperbolic plane.
The transformation $\mc T$ is the concatenation of the Cayley transform $\mc C\colon \hg \to \overline{\mathbb D}$
\[
 \mc C = \mat{1}{-i}{1}{i},
\]
the rotation $r\colon \overline{\mathbb D} \to \overline{\mathbb D}$
\[
 r \colon z \mapsto ze^{-i\frac{\pi}{2}}
\]
by $-\tfrac{\pi}{2}$, and the inverse of the Cayley transform. Hence, $\mc T$ is orientation preserving and its inverse is
\[
 \mc T^{-1} = \tfrac{1}{\sqrt 2}\mat{1}{1}{-1}{1} \colon t \mapsto \frac{1+t}{1-t}.
\]
The transformed discrete dynamical system $(E_\st, H)$, called \textit{fast discrete dynamical system} as well, is the following: We set
\begin{equation} \label{eq:defE}
\begin{aligned}
& E_\st  \sceq \mc T(D_\st), && E_{\st,k} \sceq \mc T(D_{\st,k}) \qquad\quad\text{for $k=1,\ldots, q-1$},
\\
& E_{\st,1,n} \sceq \mc T(D_{\st,1,n}), && E_{\st,q-1,n}\sceq \mc T(D_{\st,q-1,n}) \qquad\text{for $n\in\N$},
\end{aligned}
\end{equation}
and
\[
 h_k \sceq \mc T \circ g_k \circ \mc T^{-1} \quad\text{for $k=1,\ldots, q-1$.}
\]
Then $H\sceq E_\st \to E_\st$ is determined by
\[
 H\vert_{E_{\st,k}} \sceq h_k \qquad \text{for $k=2,\ldots, q-2$}
\]
and
\[
 H\vert_{E_{\st,1,n}} \sceq h_1^n, \qquad H\vert_{E_{\st,q-1,n}} \sceq h_{q-1}^n
\]
for $n\in\N$. Note that $E_\st$ is contained in the bounded set $[-1,1]$. The fast discrete dynamical system $(E_\st, H)$ arises from the cross section $\wh\CS_H\sceq \mc T(\wh\CS_G)$ with the first return map $R_H\sceq \mc T \circ R_G \circ \mc T^{-1}$. The associated transfer operator family is 
\[
 \mc L_{H,s} = \chi_{E_\st\setminus E_{\st,1}}\cdot  \sum_{n\in\N} \tau_s(h_1^n) + \sum_{k=2}^{q-2} \tau_s(h_k) + \chi_{E_\st\setminus E_{\st,q-1}}\cdot \sum_{n\in\N}\tau_s(h_{q-1}^n).
\]
If $f$ is a $\lambda$-eigenfunction of $\mc L_{G,s}$, then $\tau_s(\mc T)f$ is a $\lambda$-eigenfunction of $\mc L_{H,s}$.
We will investigate an (analytic) extension of $\mc L_{H,s}$. For this we define
\begin{align*}
E_1 &\sceq \mc T( (\lambda, \infty) )= \left( \frac{\lambda-1}{\lambda +1}, 1\right), 
\\
E_r & \sceq \mc T\left( \left(\frac{1}{\lambda}, \lambda\right)\right) = \left(-\frac{\lambda-1}{\lambda+1}, \frac{\lambda-1}{\lambda+1} \right), 
\\
E_{q-1} &\sceq \mc T\left( \left(0, \frac{1}{\lambda}\right) \right) = \left( -1 , -\frac{\lambda-1}{\lambda+1}\right)
\end{align*}
and $E\sceq E_{q-1} \cup E_r \cup E_1$. Then 
\[
 \mc L_{H,s} = \chi_{E\setminus E_1} \cdot \sum_{n\in\N} \tau_s(h_1^n) + \sum_{k=2}^{q-2} \tau_s(h_k) + \chi_{E\setminus E_{q-1}}\cdot \sum_{n\in\N}\tau_s(h_{q-1}^n)
\]
defined on a suitable subset of $\Fct(E;\C)$, which we will define in the next section. Representing each function $f\colon E\to\C$ as
\[
 \begin{pmatrix} f_1 \\ f_r \\ f_{q-1} \end{pmatrix} \sceq \begin{pmatrix} f\cdot \chi_{E_1} \\ f\cdot\chi_{E_r} \\ f\cdot\chi_{E_{q-1}} \end{pmatrix},
\]
the transfer operator $\mc L_{H,s}$ is represented by the matrix
\[
\mc L_{H,s} =
\begin{pmatrix}
0 & \sum_{k=2}^{q-2}\tau_s(h_k) & \sum_{n\in\N}\tau_s(h_{q-1}^n)
\\
\sum_{n\in\N}\tau_s(h_1^n) & \sum_{k=2}^{q-2}\tau_s(h_k) & \sum_{n\in\N}\tau_s(h_{q-1}^n)
\\
\sum_{n\in\N}\tau_s(h_1^n) & \sum_{k=2}^{q-2}\tau_s(h_k) & 0
\end{pmatrix}.
\]

\subsection{Nuclearity of the transfer operator}

In this section we will construct a domain of definition for $\mc L_{H,s}$ on which this transfer operator becomes nuclear of order $0$. We use the method of Ruelle \cite{Ruelle_zeta}, which is also used by Mayer in \cite{Mayer_thermo}. For this we have to find neighborhoods $\mc E_k$ of the sets $E_k$ in $\C$ which are mapped into each other by the fractional linear transformations that appear in $\mc L_{H,s}$. Then $\mc L_{H,s}$ is shown to map a certain nuclear space of holomorphic functions on these neighborhoods to a Banach space. As such $\mc L_{H,s}$ is nuclear of order $0$. Using a continuous embedding of the Banach space into the nuclear space, finally defines $\mc L_{H,s}$ as a nuclear self-map. 

We set
\[
 J \sceq \mc T \circ Q \circ \mc T^{-1} = \mat{-1}{0}{0}{1}.
\]
One immediately checks that $h_kJ = J h_{q-k}$  for $k=1,\ldots, q-1$.

\begin{prop}\label{choiceposs}
There exist bounded connected open subsets $\mc E_1, \mc E_r, \mc E_{q-1} \subseteq \C$ with the following properties:
\begin{enumerate}[{\rm (i)}]
\item\label{choicepossi} $\overline{E_1}\subseteq \mc E_1$, $\overline{E_r} \subseteq \mc E_r$, $\overline{E_{q-1}}\subseteq \mc E_{q-1}$,
\item\label{choicepossii} $J.\mc E_r = \mc E_r$, $J.\mc E_1 = \mc E_{q-1}$,
\item\label{choicepossiii} for $k=2,\ldots, q-2$:
\[
 h_k^{-1}.\overline{\mc E_1} \subseteq \mc E_r, \quad h_k^{-1}.\overline{\mc E_r} \subseteq \mc E_r,\quad h_k^{-1}.\overline{\mc E_{q-1}} \subseteq \mc E_r,
\]
\item\label{choicepossiv} for $n\in\N$:
\[
 h_1^{-n}.\overline{\mc E_r} \subseteq \mc E_1,\quad h_1^{-n}.\overline{\mc E_{q-1}} \subseteq \mc E_1,
\]
\item\label{choicepossv} for $n\in\N$:
\[
 h_{q-1}^{-n}.\overline{\mc E_1} \subseteq \mc E_{q-1},\quad h_{q-1}^{-n}.\overline{\mc E_r} \subseteq \mc E_{q-1},
\]
\item\label{choicepossvi} for all $z\in\mc E_1$ we have $\big|\big(h_1^{-1}\big)'(z)\big|<1$,
\item\label{choicepossvii} for all $z\in\mc E_{q-1}$ we have $\big|\big(h_{q-1}^{-1}\big)'(z)\big| < 1$,
\item\label{choicepossviii} for all $z\in\overline{\mc E_1}$ we have $\Rea z > -1$,
\item\label{choicepossix} for all $z\in\overline{\mc E_{q-1}}$ we have $1> \Rea z$,
\item\label{choicepossx} for all $z\in\overline{\mc E_r}$ we have $1 > \Rea z > -1$.
\end{enumerate}
\end{prop}

\begin{proof}
We shall provide explicit sets $\mc E_1$, $\mc E_r$, $\mc E_{q-1}$ with the required properties. To that end we set
\begin{align*}
 a_1 & \sceq -\frac{2\lambda -1}{2\lambda +1},  & a_{q-1} & \sceq - \frac{5\lambda+1}{5\lambda-1}, & a_r & \sceq - \frac{c\lambda -1}{c\lambda+1},
\\
b_1 & \sceq \frac{5\lambda +1}{5\lambda -1}, & b_{q-1} & \sceq \frac{2\lambda -1}{2\lambda +1}, & b_r & \sceq \frac{c\lambda -1}{c\lambda +1},
\end{align*}
where $c>1$ will be specified during this proof.  For $j\in \{1,r,q-1\}$ we define $\mc E_j$ to be the open Euclidean ball in $\C$ whose boundary circle passes through $a_j$ and $b_j$. Then \eqref{choicepossi}, \eqref{choicepossii} and \eqref{choicepossviii}-\eqref{choicepossx} are clearly satisfied. Each of the elements $h_1,\ldots, h_{q-1}$ is a Moebius transformation and fixes the extended real axis $\R\cup\{\infty\}$. Therefore it suffices to show that 
\begin{itemize}
\item[(\ref{choicepossiii}')] for $k=2,\ldots, q-2$, we have
\begin{align*}
& h_k^{-1}.a_1, h_k^{-1}.b_1, h_k^{-1}.1 \in \mc E_r,
\\
& h_k^{-1}.a_r, h_k^{-1}.b_r, h_k^{-1}.0 \in \mc E_r,
\\
& h_k^{-1}.a_{q-1}, h_k^{-1}.b_{q-1}, h_k^{-1}.(-1) \in \mc E_r,
\end{align*}
\item[(\ref{choicepossiv}')] for $n\in\N$ we have
\begin{align*}
& h_1^{-n}.a_r, h_1^{-n}.b_r, h_1^{-n}.0 \in \mc E_1,
\\
& h_1^{-n}.a_{q-1}, h_1^{-n}.b_{q-1}, h_1^{-n}.(-1) \in \mc E_1,
\end{align*}
\item[(\ref{choicepossv}')] for $n\in\N$ we have
\begin{align*}
& h_{q-1}^{-n}.a_1, h_{q-1}^{-n}.b_1, h_{q-1}^{-n}.1 \in \mc E_{q-1},
\\
& h_{q-1}^{-n}.a_r, h_{q-1}^{-n}.b_r, h_{q-1}^{-n}.0 \in \mc E_{q-1}.
\end{align*}
\end{itemize}
Let $k\in 2,\ldots, q-2$. Since $h_k^{-1}.E \subseteq E_r$, we clearly have 
\[
 h_k^{-1}.a_1, h_k^{-1}.a_r, h_k^{-1}.b_r, h_k^{-1}.0, h_k^{-1}.b_{q-1} \in \mc E_r
\]
independent of the choice of $c>1$.
Further
\[
 h_k^{-1}.1, h_k^{-1}.(-1) \in \left[  -\frac{\lambda -1}{\lambda+1}, \frac{\lambda-1}{\lambda+1}\right].
\]
Since $c>1$, we have $h_k^{-1}.1, h_k^{-1}.(-1) \in \mc E_r$ for any choice of $c$. In the following we show that $1 > h_k^{-1}.b_1, h_k^{-1}.a_{q-1} > -1$. If we use the short notation $\xi_k \sceq \sin\big(\tfrac{k}{q}\pi\big)$,
then 
\[
 h_k^{-1} = \frac{1}{2\sin\frac{\pi}{q}}
\mat{ 2\xi_k - \xi_{k-1} - \xi_{k+1}}{ -\xi_{k-1} + \xi_{k+1}}
{+\xi_{k-1} - \xi_{k+1}}{ 2\xi_k + \xi_{k-1} + \xi_{k+1} },
\]
and hence we have
\[
 h_k^{-1}.a_{q-1} = \frac{5\lambda\left( \xi_k - \xi_{k+1}\right) - \xi_{k-1} + \xi_k}
{-5\lambda\left(\xi_k + \xi_{k+1}\right) + \xi_{k-1} + \xi_{k}}.
\]

By addition theorems we get 
\begin{align*}
 -5\lambda& \left( \sin\big(\tfrac{k}{q}\pi\big) + \sin\big(\tfrac{k+1}{q}\pi\big)\right) + \sin\big(\tfrac{k-1}{q}\pi\big) + \sin\big(\tfrac{k}{q}\pi\big) 
\\
& = \left(-5\lambda\sin\big(\tfrac{k}{q}\pi\big) + \sin\big(\tfrac{k-1}{q}\pi\big) \right) + \left( -5\lambda\sin\big(\tfrac{k+1}{q}\pi\big) + \sin\big(\tfrac{k}{q}\pi\big) \right) 
\\
& < 0.
\end{align*}
Therefore, $h_k^{-1}.a_{q-1} > -1$ if and only if
\[
 -5\lambda\sin\big(\tfrac{k+1}{q}\pi\big) + \sin\big(\tfrac{k}{q}\pi\big) < 0,
\]
which is satisfied. Similarly, we see that $h_k^{-1}.a_{q-1}<1$. Then 
\[
 h_k^{-1}.b_1 = h_k^{-1}.\big(J.a_{q-1}\big) = J.\big(h_{q-k}^{-1}.a_{q_1}\big) \in (-1,1).
\]
Hence we can choose $c>1$ such that 
\[
 -\frac{c\lambda-1}{c\lambda+1} < h_k^{-1}.a_{q-1}, h_k^{-1}.b_1 < \frac{c\lambda-1}{c\lambda+1}
\]
for $k\in\{2,\ldots, q-2\}$. This proves (\ref{choicepossiii}').

For $n\in\N$ we have $h_1^{-n}(E) \subseteq \overline{E_1} \subseteq \mc E_1$. Thus, 
\[
 h_1^{-n}.(-1), h_1^{-n}.a_r, h_1^{-n}.0, h_1^{-n}.b_{q-1}, h_1^{-n}.b_r \in \mc E_1.
\]
It remains to show that $h_1^{-n}.a_{q-1}\in\mc E_1$. We have
\[
 h_1^{-n}.a_{q-1} = \frac{5n\lambda^2 - 5\lambda -1}{5n\lambda^2 + 5\lambda -1} \quad < 1 < b_1.
\]
Further $a_1 < h_1^{-n}.a_{q_1}$ is equivalent to 
\[
 0 < 2\lambda ( 10n\lambda^2 - 7),
\]
which is obviously  true. Hence, (\ref{choicepossiv}') is satisfied. Then (\ref{choicepossv}') follows from (\ref{choicepossiv}') using the symmetry $J$.

In order to show \eqref{choicepossvi}, let $z\in\mc E_1$. Then $\big| \big(h_1^{-1}\big)'(z)\big| < 1$ if and only if
\[
 4-2\lambda^2 < |-\lambda z + 2 + \lambda|^2.
\]
If we use $z=x+iy$ with $x,y\in\R$, it suffices to show that
\[
 4-2\lambda^2 < (2+\lambda - \lambda x)^2.
\]
From
\[
 x< b_1 = \frac{5\lambda+1}{5\lambda-1}
\]
and
\[
 b_1 < \frac{2+\lambda}{\lambda} 
\]
it follows that
\[
 (2+\lambda - \lambda x)^2 > \left(2+\lambda - \lambda b_1 \right)^2 = \left(\frac{8\lambda - 2}{5\lambda -1}\right)^2.
\]
We now claim that 
\[
 \left(\frac{8\lambda-2}{5\lambda-1}\right)^2 > 4-2\lambda^2.
\]
This is equivalent to
\[
 50 \lambda^3 - 20\lambda^2 - 34\lambda + 8 > 0,
\]
which is the case since $\lambda\geq 1$. In turn, $\big| \big(h_1^{-1}\big)'(z)\big| < 1$. 

Let $w\in\mc E_{q-1}$. Then $z\sceq J.w \in \mc E_1$. Further
\begin{align*}
1 & > \big|\big(h_1^{-1}\big)'(z)\big| = \big| \big(h_{q-1}^{-1}J\big)'(Jw)\big| = \big| \big(h_{q-1}^{-1}\big)'(w)\big|.
\end{align*}
Thus, \eqref{choicepossvii} is satisfied. This completes the proof.
\end{proof}

From now on we fix a choice of sets $\mc E_1$, $\mc E_r$, $\mc E_{q-1}$ with the properties of Proposition~\ref{choiceposs}.
For $j\in\{1,r,q-1\}$, we let
\[
 \mc H(\mc E_j) \sceq \{ f\colon \mc E_j \to \C \text{ holomorphic}\}
\]
and 
\[
 \mc H(\mc E) \sceq \mc H(\mc E_1) \times \mc H(\mc E_r) \times \mc H(\mc E_{q-1}).
\]
Endowed with the compact-open topology, each $\mc H(\mc E_j)$ is a nuclear space \cite[I, \S 2, no.\@ 3, Corollary]{Grothendieck_produit}. Hence the product space $\mc H(\mc E)$ is nuclear \cite[I, \S 2, no.\@ 2, Th\'eor\`eme 9]{Grothendieck_produit}.
Moreover we set
\[
 B(\mc E_j) \sceq \{ f\in \mc H(\mc E_j)\mid \text{$f$ extends continuously to $\overline{\mc E_j}$} \}.
\]
Endowed with the supremum norm, $B(\mc E_j)$ is a Banach space. Further we let
\begin{equation}\label{banachspace}
 B(\mc E) \sceq B(\mc E_1) \times B(\mc E_r) \times B(\mc E_{q-1})
\end{equation}
be the direct product of the Banach spaces $B(\mc E_1)$, $B(\mc E_r)$ and $B(\mc E_{q-1})$. We let $\mc L_{H,s}$ act on $\mc H(\mc E)$ by its matrix representation.

\begin{prop}\label{maps}
Let $s\in\C$ with $\Rea s > \tfrac12$. Then the transfer operator $\mc L_{H,s}$ maps $\mc H(\mc E)$ to $B(\mc E)$.
\end{prop}

\begin{proof}
Let $f\in \mc H(\mc E)$, hence $f=(f_1, f_r, f_{q-1})^\top$ with $f_j\in \mc H(\mc E_j)$. Then $b\sceq \mc L_{H,s}f$ with
\begin{align}
b_1 & =  \sum_{k=2}^{q-2} \tau_s(h_k)f_r + \sum_{n\in\N}\tau_s(h_{q-1}^n)f_{q-1},
\\
b_r & =  \sum_{n\in\N}\tau_s(h_1^n)f_1 + \sum_{k=2}^{q-2}\tau_s(h_k)f_r + \sum_{n\in\N}\tau_s(h_{q-1}^n)f_{q-1},
\\
b_{q-1} & = \sum_{n\in\N}\tau_s(h_1^n)f_1 + \sum_{k=2}^{q-2}\tau_s(h_k)f_r.
\end{align}
We have to show that $b_1\in B(\mc E_1)$, $b_r\in B(\mc E_r)$ and $b_{q-1}\in B(\mc E_{q-1})$. 

At first let $j\in\{1,r,q-1\}$ and $k\in\{2,\ldots, q-2\}$. Clearly, $\tau_s(h_k)f_r$ is holomorphic on $\mc E_j$. Since $h_k^{-1}.\overline{\mc E_j}\subseteq \mc E_r$ by hypothesis, 
\[
 \tau_s(h_k)f_r\colon z\mapsto j_s(h_k^{-1},z)f_r(h_k^{-1}.z)
\]
is well-defined and continuous on $\overline{\mc E_j}$. Thus, the summand $\sum_{k=2}^{q-2}\tau_s(h_k)f_k$ defines a map which is holomorphic on $\mc E_j$ and extends continuously to $\overline{\mc E_j}$. 

Now we consider $\sum_{n\in\N}\tau_s(h_{q-1}^n)f_{q-1}$ on $\mc E_1$ and on $\mc E_r$. Let $j\in\{1,r\}$ and $n\in\N$. Since $h_{q-1}^{-1}.\overline{\mc E_j} \subseteq \mc E_{q-1}$ and $\big|\big(h_{q-1}^{-1}\big)'(z)\big| < 1$ for all $z\in\mc E_{q-1}$ by hypothesis, we have 
\[
 h_{q-1}^{-n}.\overline{\mc E_j} \subseteq h_{q-1}^{-1}.\overline{\mc E_j} \subseteq \mc E_{q-1}.
\]
In turn, because $f_{q-1}$ is continuous on $\mc E_{q-1}$,  there exists $K\geq 0$ such that
\[
 \sup\left\{ \big| f_{q-1}(h_{q-1}^{-n}.z) \big| \mid z\in\overline{\mc E_j},\ n\in\N \right\} \leq K.
\]
Moreover, the map $\tau_s(h_{q-1}^{n})f_{q-1}$ is continuous on $\overline{\mc E_j}$.
Further
\[
 j_s\big(h_{q-1}^{-n}, z\big) = \left(\frac{4}{(n\lambda z + 2 +n\lambda)^2} \right)^s.
\]
Thus
\[
 \big| j_s\big(h_{q-1}^{-n}, z\big) \big| \leq \frac{4^{\Rea s} e^{\pi |\Ima s|}}{|n\lambda z + 2 + n\lambda|^{2\Rea s}} \leq \frac{4^{\Rea s} e^{\pi |\Ima s|}}{|n\lambda(x+1) + 2|^{2\Rea s}}
\]
where $x \sceq \Rea z$. For $z\in \overline{\mc E_j}$ we have $x = \Rea z > -1$. Let 
\[
 x_0 \sceq \min\{ \Rea z \mid z\in \overline{\mc E_j}\}.
\]
Then
\[
  \big| j_s\big(h_{q-1}^{-n}, z\big) \big| \leq \frac{4^{\Rea s} e^{\pi |\Ima s|}}{\big(n\lambda(x_0 + 1) + 2\big)^{2\Rea s}}.
\]
In turn,
\[
 \sup_{z\in\overline{\mc E_j}} \big|\tau_s(h_{q-1}^n) f_{q-1}(z)\big| \leq \frac{4^{\Rea s} e^{\pi |\Ima s|}}{\big(n\lambda(x_0 + 1) + 2\big)^{2\Rea s}}\cdot K.
\]
This shows that the series $\sum_{n\in\N}\tau_s(h_{q-1}^n)f_{q-1}$ converges uniformly on $\overline{\mc E_j}$ if $\Rea s > \tfrac12$. The Weierstrass Theorem shows that $\sum_{n\in\N}\tau_s(h_{q-1}^n)f_{q-1}$ is continuous on $\overline{\mc E_j}$ and  holomorphic on $\mc E_j$.

Analogously one sees that $\sum_{n\in\N}\tau_s(h_1^n)f_1$ is continuous on $\overline{\mc E_r}$ and on $\overline{\mc E_{q-1}}$, and holomorphic on $\mc E_r$ and on $\mc E_{q-1}$. This completes the proof.
\end{proof}

\begin{prop}\label{nuclear1}
If $\Rea s > \tfrac12$,  the map $\mc L_{H,s}\colon \mc H(\mc E) \to B(\mc E)$ is nuclear of order $0$.
\end{prop}

\begin{proof}
According to \cite[II, \S 2, no.\@ 1, Corollaire 4]{Grothendieck_produit} it suffices to show that the image of some neighborhood of $0$ in $\mc H(\mc E)$ under $\mc L_{H,s}$ is bounded in $B(\mc E)$. We pick compact balls $K_1\subseteq \mc E_1$, $K_r \subseteq \mc E_r$, $K_{q-1}\subseteq \mc E_{q-1}$, each of which contains more than one point. Let $M>0$. Then
\[
V_M \sceq \big\{ f\in \mc H(\mc E) \ \big\vert\  \sup_{z\in K_j} |f_j(z)| < M,\ j\in\{1,r,q-1\} \big\}
\]
is a neighborhood of $0$ in $\mc H(\mc E)$. Along the lines of the proof of Proposition~\ref{maps} we deduce that 
\[
 \sup_{f\in V_M} \| \mc L_{H,s}f\|_\infty
\]
is bounded. Now \cite[II, \S 2, no.\@ 4, Corollaire 2]{Grothendieck_produit} shows that $\mc L_{H,s}$ is nuclear of order $0$.
\end{proof}

\begin{prop}\label{isnuclear}
Let $\Rea s > \tfrac12$. Then the transfer operator $\mc L_{H,s}$ is a self-map of $B(\mc E)$ and as such nuclear of order $0$.
\end{prop}

\begin{proof}
The embedding $j\colon B(\mc E) \to \mc H(\mc E)$, $f\mapsto f$, is linear and continuous. By \cite[II, p.\@ 9]{Grothendieck_produit}, Proposition~\ref{nuclear1} implies that $\mc L_{H,s} \colon B(\mc E) \to B(\mc E)$ is nuclear of order $0$.
\end{proof}

\subsection{The Fredholm determinant of $\mc L_{H,s}$}

For $\Rea s > \tfrac12$ the Fredholm determinant of $\mc L_{H,s}$ is defined as 
\[
 \det( 1 - \mc L_{H,s} ) \sceq \exp\left( -\sum_{n=1}^\infty \frac1n \Tr \mc L_{H,s}^n \right).
\]
In order to show that it equals the Selberg zeta function (for $\Rea s > 1$), we take advantage of the following argument by Ruelle \cite{Ruelle_dynzeta}, linking the Selberg zeta function (defined via a continuous dynamical system, namely the geodesic flow) to dynamical partition functions (defined via an appropriate discrete dynamical system).  We may apply it in our situation because each closed geodesic on $Y$ intersects the cross section $\wh\CS_H$ and because his arguments holds \textit{verbatim} if the fixed point sets of the appropriate discrete dynamical system are countable (his original argument asks for finite fixed point sets).

Recall that $P_\dyn$ denotes the set of periodic geodesics on $Y$, and that $l(\wh\gamma)$ is the length of $\wh\gamma \in P_\dyn$. For $\Rea s > 1$ let 
\[
 \zeta_{SR}(s) \sceq \prod_{\wh\gamma \in P_\dyn} \left( 1 - e^{-sl(\wh\gamma)}\right)^{-1}
\]
denote the \textit{Smale-Ruelle zeta function}. Then the Selberg zeta function becomes
\[
 Z(s) = \prod_{k=0}^\infty \zeta_{SR}(s+k)^{-1}.
\]
Both zeta functions converge for $\Rea s > 1$ (\cite{Selberg, Fischer}; note that their convergence is equivalent), and hence this equality holds on the level of converging products. For $\wh v\in\wh\CS_H$ let $\wh\gamma_{v}$ denote the geodesic on $Y$ determined by $\wh v$ (ie., $\wh\gamma_{v}'(0) = \wh v$). Then
\[
 r_H(\hat v) \sceq \min\big\{ t>0\ \big\vert\  \wh\gamma_{v}'(t) \in \wh\CS_H \big\}
\]
is the \textit{first return time} of $\wh v$, which exists at least if $\wh\gamma_{v}$ is closed. Further, for $n\in \N$ let
\[
 \Fix R_H^n \sceq \big\{ \wh v \in \wh\CS_H\ \big\vert\ R^n_H(\wh v) = \wh v \big\}
\]
be the fixed point set of $R_H^n$. For $n\in \N$ and $s\in \C$,  the \textit{$n$-th dynamical partition function} is defined as 
\[
 Z_n(R_H,s) \sceq \sum_{\wh v\in \Fix R_H^n} \prod_{k=0}^{n-1} \exp\left( -s r_H\big(R^k_H(\hat v)\big) \right) \qquad\text{(as formal series).}
\]
Then \cite{Ruelle_dynzeta} we have the formal identity
\[
\zeta_{SR}(s) = \exp\left(\sum_{n=1}^\infty \frac1n Z_n(R_H,s) \right).
\]

In the following we will evaluate the traces $\Tr \mc L_{H,s}^n$ and the dynamical partition functions $Z_n$ in terms of certain elements of $G_q$. This will allow us to show that 
\[
 Z_n(R_H,s) = \Tr \mc L_{H,s}^n - \Tr \mc L_{H,s+1}^n,
\]
which is the crucial step towards the equality of the Fredholm determinant of $\mc L_{H,s}$ and the Selberg zeta function. The proofs rely on key properties of the slow discrete dynamical system and the relation between the coding sequences associated to the slow and the fast discrete dynamical system, which we are defining now.

Let $\pi \colon S\h \to G_q\backslash S\h$ denote the quotient map of the unit tangent bundles. For $\wh v\in \wh\CS_F$ let $v\sceq \pi^{-1}(\wh v) \cap \CS'_F$. Let $\gamma_{v}$ denote the geodesic on $\h$ determined by $\gamma'_{v}(0) = v$. The \textit{forward $F$-coding sequence} of $\wh v$ is given by the sequence $(a_n)_{n\in\N_0}$ where $a_n \sceq g_k^{-1}$  if and only if $F^n(\gamma_{v}(\infty)) \in D_{\st,k}$ with $k\in\{1,\ldots, q-1\}$. The set of arising forward $F$-coding sequences is denoted by $\Lambda_F$.

Analogously, for $\wh v\in \wh\CS_H$ we let $v\sceq \pi^{-1}(\wh v) \cap \CS'_H$. The \textit{forward $H$-coding sequence} of $\wh v$ is the sequence $(a_n)_{n\in\N_0}$ where
\begin{align*}
 a_n&\sceq h_k^{-1} &\Leftrightarrow && H^n(\gamma_{v}(\infty)) &\in E_{\st,k}\ \text{and}\ k\in\{2,\ldots, q-2\},
\\
a_n & \sceq h_1^{-m} & \Leftrightarrow && H^n(\gamma_{v}(\infty)) &\in E_{\st,1,m}\ \text{and}\ m\in\N,
\\
a_n & \sceq h_{q-1}^{-m} & \Leftrightarrow && H^n(\gamma_{v}(\infty)) &\in E_{\st,q-1,m}\ \text{and}\ m\in\N.
\end{align*}
We denote the set of arising forward $H$-coding sequences by $\Lambda_H$. 

Periodic forward coding sequences (in which we are mostly interested) correspond in the obvious way to coding sequences (cf.\@ \cite{Pohl_Symdyn2d}), for which reason we will omit ``forward'' when refering to periodic forward coding sequences. A periodic coding sequence is denoted by
\[
 (\overline{a_0,\ldots, a_{k-1}})
\]
where $a_0,\ldots, a_{k-1}$ is a (not necessarily minimal) period of the coding sequence $(a_n)_{n\in \N_0}$.

For $n\in\N$ we define
\begin{align*}
\Per_{H,n} & \sceq \{ (\overline{a_0,\ldots, a_{n-1}})\in\Lambda_H \}, \text{ and}
\\
P_{H,n} &\sceq \{ (a_0\cdots a_{n-1})^{-1} \mid (\overline{a_0,\ldots, a_{n-1}})\in\Lambda_H\}.
\end{align*}

Propositions~\ref{bij1} and \ref{bij2} below show that these sets and $\Fix R_H^n$ are in natural bijections. Moreover they characterize the elements in $P_{H,n}$. 

\begin{prop}\label{bij1}
For each $n\in\N$, the map
\[
\left\{
\begin{array}{ccl}
\Fix R^n_H & \to & \Per_{H,n}
\\
\wh v & \mapsto & \text{$H$-coding sequence of $\wh v$}
\end{array}
\right.
\]
is a bijection.
\end{prop}

\begin{proof}
For each $m\in\N$, the map 
\[
\left\{
\begin{array}{ccl}
\Fix R^m_F & \to & \{ (\overline{b_0,\ldots, b_{m-1}}) \in \Lambda_F\}
\\
\wh v & \mapsto & \text{$F$-coding sequence of $\wh v$}
\end{array}
\right.
\]
is a bijection by \cite[Corollary~4.158]{Pohl_Symdyn2d}. Suppose that $\wh v\in \Fix R_F^m$ is such that $\wh w \sceq \mc T(\wh v) \in \wh\CS_H$. Let $(\overline{b_0,\ldots, b_{m-1}})$ be the $F$-coding sequence of $\wh v$. Then the $H$-coding sequence $(a_k)_{k \in \N_0}$ of $\wh w$ is periodic and it arises from $(b_0,\ldots, b_{m-1})$ in the following way: For $k=0,\ldots, m-1$ set $c_k\sceq \mc T\circ b_k \circ \mc T^{-1}$. Now collapse the successive appearances of $h_1^{-1}$ resp.\@ $h_{q-1}^{-1}$ in $c_0,\ldots, c_{m-1}$ to $h_1^{-n}$ resp.\@ $h_{q-1}^{-n}$ ($n$ being the number of successive $h_1^{-1}$ resp.\@ $h_{q-1}^{-1}$). This gives 
\[
 a_0,\ldots, a_{p-1}
\]
for a (unique) $p\leq m$. Then 
\[
 (a_k)_{k \in \N_0} = (\overline{a_0,\ldots, a_{p-1}}).
\]
Note that $m$ uniquely determines $p$, and that this construction is invertible. Since each periodic $H$-coding sequence arises in this way, the bijections above imply the claimed bijections between $\Fix R_H^n$ and $\Per_{H,n}$.
\end{proof}

\begin{lemma}\label{freesg}
The sub-semigroup of $\mc T \circ G_q \circ \mc T^{-1}$ generated by $\{ h_1,\ldots, h_{q-1} \}$ is free. Its parabolic elements are $\{ h_1^n, h_{q-1}^n\mid n\in\N\}$. All other elements are hyperbolic.
\end{lemma}

\begin{proof}
The statement is equivalent to that the sub-semigroup $S_q$ of $G_q$ generated by $\{g_1^{-1},\ldots, g_{q-1}^{-1}\}$ is free, its parabolic elements are $\{g_1^{-n}, g_{q-1}^{-n} \mid n\in\N\}$ and all other elements are hyperbolic.  Let $g= g_{l_1}^{-1}\cdots g_{l_k}^{-1}$ be an element of $S_q$ such that not all $l_j=1$ and not all $l_j=q-1$. An elementary calculation shows that $g$ has only positive entries. By the Perron-Frobenius Theorem, $g$ has a real positive eigenvalue of algebraic multiplicity $1$. Hence $g$ is hyperbolic. Then $g$ fixes a unique geodesic $\gamma$ on $H$ with $\gamma(\infty) > 0$. By \cite[Propositions~4.182, 4.183]{Pohl_Symdyn2d}, the $F$-coding sequence associated to $\gamma(\infty)$ is unique and has the (not necessarily minimal) period $(g_{l_1}^{-1},\ldots, g_{l_k}^{-1})$. Moreover, $g_{l_1}^{-1}\cdots g_{l_k}^{-1}$ is the unique presentation of $g$ by the given set of generators of $S_q$. In turn, the presentation of any hyperbolic element in $S_q$ is unique. The only non-hyperbolic elements are of the form  
\[
 \mat{1}{n\lambda}{0}{1} \quad\text{and}\quad \mat{1}{0}{n\lambda}{1}\quad\text{with $n\in\N$},
\]
of which the presentation is obviously unique.
\end{proof}

Let $n\in \N$ and suppose that $w=s_0\ldots s_{n-1}$ is a word in (the symbols) 
\[
\{h_1^m,h_2,\ldots, h_{q-2}, h_{q-1}^m\mid m\in\N\}.
\]
We call $w$ \textit{reduced} if it does not contain a subword of the form $h_1^{m_1}h_1^{m_2}$ or $h_{q-1}^{m_1}h_{q-1}^{m_2}$. We say that $w$ is \textit{regular} if it is reduced and also $ww$ is reduced. The \textit{length} of $w$ is $n$.

An immediate consequence of Lemma~\ref{freesg} is the following observation.

\begin{lemma}\label{reducedwords}
The map
\[
\left\{
\begin{array}{ccl}
 \{\text{reduced words}\} & \to & \mc T\circ G_q \circ \mc T^{-1}
\\
s_0\ldots s_{n-1} & \mapsto & s_0\cdots s_{n-1}
\end{array}
\right.
\]
is injective. Further, the image of each regular word is a hyperbolic element.
\end{lemma}

For convenience, if $s_0\ldots s_{n-1}$ is a reduced resp.\@ regular word of length $n$, then we also call the group element $s_0\cdots s_{n-1}$ a reduced resp.\@ regular word of length $n$.

\begin{prop}\label{bij2}
For each $n\in\N$, the map
\[
\left\{
\begin{array}{ccl}
 \Per_{H,n} & \to & P_{H,n}
\\
(\overline{a_0,\ldots, a_{n-1}}) & \mapsto & (a_0\cdots a_{n-1})^{-1}
\end{array}
\right.
\]
is a bijection. Moreover, $P_{H,n}$ is the set of regular words of length $n$.
\end{prop}

\begin{proof}
By Lemma~\ref{reducedwords} it suffices to show that 
\[
 \Per_{H,n} = \{ (\overline{a_0,\ldots, a_{n-1}}) \mid \text{$(a_0\cdots a_{n-1})^{-1}$ is a regular word of length $n$}\}.
\]
One easily sees that the set of all periodic $F$-coding sequences is the set of all periodic sequences $(\overline{b_0,\ldots, b_{m-1}})$ with $m\in\N$ such that $b_j \in \{g_1^{-1},\ldots, g_{q-1}^{-1}\}$ for all $j\in\{0,\ldots, m-1\}$ but not all $b_j=g_1^{-1}$ and not all $b_j=g_{q-1}^{-1}$. This proves the claim.
\end{proof}

Let $g\in\PSL(2,\R)$ be hyperbolic. Among the two fixed points of $g$, one is attractive (for the iterated action of $g$) and the other one is repelling. We denote the attractive point by $z^*(g)$, the repelling by $w^*(g)$. Further there exists (a unique) $\lambda(g) \in \R$, $|\lambda(g)|>1$, such that $g$ is conjugate in $\PGL(2,\R)$ to 
\[
\mat{\lambda(g)}{}{}{\lambda(g)^{-1}}.
\]

\begin{lemma}\label{attractive}
Let $g\in \PSL(2,\R)$ be a hyperbolic matrix and suppose that $\wt g = \textmat{a}{b}{c}{d}$ is a representative of $g$ in $\SL(2,\R)$. Then 
\[
 |\lambda(g)| = \frac{|\Tr g| + \sqrt{(\Tr g)^2 - 4}}{2},
\]
and
\[
 z^*(g) = \frac{\lambda(g) - d}{c} = \frac{1}{w^*(g)}
\]
where one has to take $\lambda(g) >0$ if and only if $\Tr \wt g > 0$. The last formulas are understood as $z^*(g) = \infty$ and $w^*(g) = 0$ if $c=0$.
Further 
\[
 g'\big(z^*(g)\big) = \lambda(g)^{-2}\quad\text{and}\quad g'\big(w^*(g)\big) = \lambda(g)^2.
\]
\end{lemma}

\begin{proof}
 This is a straightforward calculation.
\end{proof}

\begin{prop}\label{tracerepr}
For $\Rea s > \tfrac12$ and $n\in\N$ we have 
\[
\Tr \mc L_{H,s}^n = \sum_{a\in P_{H,n}} \frac{ \big(a'(z^*(a))\big)^s }{1-a'(z^*(a))}.
\]
\end{prop}

\begin{proof}
The set of reduced words of length $n$ decomposes into the subsets
\begin{enumerate}
\item[$A_{(1,1)}^n$] of reduced words of length $n$ which begin and end with an element of $\{ h_1^m \mid m\in\N\}$, 
\item[$A_{(1,r)}^n$] of reduced words of length $n$ which begin with an element of $\{h_1^m\mid m\in\N\}$ and end with an element of $\{h_2,\ldots, h_{q-2}\}$,
\item[$A_{(1,q-1)}^n$] of reduced words of length $n$ which begin with an element of $\{h_1^m \mid m\in\N\}$ and end with an element of $\{ h_{q-1}^m \mid m\in \N\}$,
\end{enumerate}
and the sets $A_{(r,1)}^n$, $A_{(r,r)}^n$, $A_{(r,q-1)}^n$, $A_{(q-1,1)}^n$, $A_{(q-1,r)}^n$ and $A_{(q-1,q-1)}^n$ defined in the obvious way. Let
\begin{align*}
A_1^n & \sceq A_{(1,1)}^n \cup A_{(r,1)}^n \cup A_{(q-1,1)}^n,
\\
A_r^n & \sceq A_{(1,r)}^n \cup A_{(r,r)}^n \cup A_{(q-1,r)}^n, \text{ and}
\\
A_{q-1}^n & \sceq A_{(1,q-1)}^n \cup A_{(r,q-1)}^n \cup A_{(q-1,q-1)}^n.
\end{align*}

Then 
\[
\mc L_{H,s}^n = 
\begin{pmatrix}
\sum\limits_{a\in A_1^n\setminus A_{(1,1)}^n} \tau_s(a)  & \sum\limits_{a\in A_r^n\setminus A_{(1,r)}^n} \tau_s(a) & \sum\limits_{a\in A_{q-1}^n\setminus A_{(1,q-1)}^n} \tau_s(a) 
\\
\sum\limits_{ a\in A_1^n} \tau_s(a) & \sum\limits_{a\in A_r^n} \tau_s(a) & \sum\limits_{a\in A_{q-1}^n} \tau_s(a) 
\\
\sum\limits_{a\in A_1^n\setminus A_{(q-1,1)}^n} \tau_s(a) & \sum\limits_{a\in A_r^n\setminus A_{(q-1,r)}^n} \tau_s(a) & \sum\limits_{a\in A_{q-1}^n\setminus A_{(q-1,q-1)}^n} \tau_s(a)
\end{pmatrix}.
\]
By Proposition~\ref{bij2}, 
\[
 P_{H,n} = A_{(r,1)}^n \cup A_{(q-1,1)}^n \cup A_r^n \cup A_{(1,q-1)}^n \cup A_{(r,q-1)}^n.
\]
By \cite[Chapitre II.2, Prop.~2]{Grothendieck_fredholm},
\[
 \Tr \mc L_{H,s}^n = \Tr\sum_{a\in P_{H,n}}\tau_s(a).
\]
There exist open bounded subsets $\mc F_1, \mc F_r, \mc F_{q-1} \subseteq \C$ which satisfy Proposition~\ref{choiceposs} and in addition $\overline{\mc F_j} \subseteq \mc E_j$ for $j\in\{1,r,q-1\}$. Let 
\[
 B(\mc F) \sceq B(\mc F_1) \times B(\mc F_r) \times B(\mc F_{q-1})
\]
be the Banach space defined as in \eqref{banachspace}. Further let
\[
r\colon\left\{
\begin{array}{ccl}
B(\mc E) & \to & B(\mc F)
\\
(f_1,f_r,f_{q-1})& \mapsto & (f_1\vert_{\mc F_1}, f_r\vert_{\mc F_r}, f_{q-1}\vert_{\mc F_{q-1}})
\end{array}
\right.
\]
denote the restriction map. W.l.o.g.\@ we may assume (possible after applying Riemann mappings) that for $j\in\{1,r,q-1\}$ there exist $0<\sigma_j<\varrho_j$ and $z_j\in\C$ such that 
\[
 \overline{\mc F_j} \subseteq B_{\sigma_j}(z_j) \subseteq B_{\varrho_j}(z_j) \subseteq \mc E_j.
\]
For all $f_j\in B(\mc E_j)$ we have
\[
 r(f_j)(z) = \frac1{2\pi i} \sum_{k=0}^\infty \int\limits_{\ |\zeta - z_j|=\varrho_j} \frac{f_j(\zeta)}{(\zeta - z_j)^{k+1}}\, d\zeta \cdot (z-z_j)^k.
\]
For $k\in\N_0$ let 
\begin{align*}
e_{j,k}&\colon\left\{
\begin{array}{ccl}
\mc F_j & \to & \C
\\
z & \mapsto & (z-z_j)^k
\end{array}
\right.,
\\[2mm]
e_k &\sceq (e_{1,k}, e_{r,k}, e_{q-1,k}) \in B(\mc F)
\end{align*}
and
\begin{align*}
 \varphi_{j,k} &\colon\left\{
\begin{array}{ccl}
B(\mc E_j) & \to & \C
\\
f_j & \mapsto & \frac{1}{2\pi i} \int\limits_{|\zeta-z_j| = \varrho_j} \frac{f_j(\zeta)}{(\zeta-z_j)^{k+1}}\, d\zeta
\end{array}
\right.
\\[2mm]
\varphi_k &\sceq (\varphi_{1,k}, \varphi_{r,k}, \varphi_{q-1,k}) \in B(\mc E)'.
\end{align*}
Then 
\[
 r = \sum_{k\in\N_0} \varphi_k \otimes e_k
\]
is seen to be nuclear of order $0$. Moreover,
\[
 \sum_{a\in P_{H,n}} \tau_s(a) = \sum_{k\in\N_0} \sum_{a\in P_{H,n}} \varphi_k \otimes \tau_s(a)e_k
\]
where $\varphi_k\otimes\tau_s(a)e_k$ is understood as $\varphi_{1,k} \otimes \tau_s(a)e_{1,k}$ for $a\in A_1^n\setminus A^n_{(1,1)}$ (and so on).
Hence
\[
 \Tr \mc L_{H,s}^n = \Tr \sum_{a\in P_{H,n}}\tau_s(a) = \sum_{a\in P_{H,n}}\sum_{k\in\N_0} \varphi_k(\tau_s(a)e_k) = \sum_{a\in P_{H,n}}\Tr \tau_s(a).
\]
Now \cite{Ruelle_zeta} shows
\[
 \Tr \tau_s(a) = \frac{ \big(a'(z^*(a))\big)^s }{1-a'(z^*(a))}.
\]
This completes the proof.
\end{proof}

\begin{thm}
For $\Rea s > 1$, we have $Z(s) = \det(1-\mc L_{H,s})$.
\end{thm}

\begin{proof}
Let $n\in \N$ and $\wh v\in \Fix R_H^n$. Suppose that $a$ is the element in $P_{H,n}$ which corresponds to $\wh v$ by Propositions~\ref{bij1} and \ref{bij2}. Let
\[
 m\sceq \min \big\{ h\in\N\ \big\vert\ \wh v \in \Fix R^h_H \big\}.
\]
and let $b$ be the element in $P_{H,m}$ which corresponds to $\wh v$. Suppose that $\wh\gamma$ is the closed geodesic on $Y$ determined by $\wh\gamma'(0) = \wh v$. Then \cite[Proposition~4.153]{Pohl_Symdyn2d} yields that $b$ is the hyperbolic matrix which determines $\wh\gamma$, that means any other hyperbolic matrix which fixes the endpoints $\gamma(\pm\infty)$ for any representative $\gamma$ of $\wh\gamma$ on $\h$ is a power of $b$. It is well-known that then $l(\wh\gamma)=\log N(b)$ where $N(b)=\lambda(b)^2$ is the norm of $b$. By Lemma~\ref{attractive},
\[
 l(\wh\gamma) = -\log b'(z^*(b)).
\]
Further, $p\sceq n/m$ is a positive integer and $a=b^p$. Thus,
\[
 \sum_{k=0}^{n-1} r_H(R_H^k \wh v) = p l(\wh\gamma) = -\log a'(z^*(a)).
\]
Therefore,
\[
 Z_n(R_H, s) = \sum_{a\in P_{H,n}} \big( a'(z^*(a)) \big)^s \qquad\text{(formally).}
\]
Thus, by Proposition~\ref{tracerepr}
\[
 Z_n(R_H,s) = \Tr \mc L_{H,s}^n - \Tr \mc L_{H,s+1}^n.
\]
This shows in addition that $Z_n(R_H,s)$ converges for $\Rea s > \tfrac12$. Now the claim follows as in \cite{Mayer_thermoPSL}.
\end{proof}

\subsection{Meromorphic continuation}

From Proposition~\ref{tracerepr} it follows that the Fredholm determinant $\det(1-\mc L_{H,s})$ is holomorphic in $\{ \Rea s > \tfrac12\}$. In this section we will show that $\det(1-\mc L_{H,s})$ extends to a meromorphic function on $\C$ with possible poles at $s=(1-k)/2$, $k\in\N_0$. The method of proof is adapted from \cite{Mayer_thermo} and \cite{Morita_transfer}.
 
We start by proving (Proposition~\ref{meroext} below) that the map $s\mapsto \mc L_{H,s}$ extends to a meromorphic function on $\C$ with values in nuclear operators of order $0$. This means that we find a discrete set $P\subseteq\C$ (candidates for poles) and for each $s\in\C\setminus P$ we find a nuclear operator $\wt{\mc L}_{H,s} \colon B(\mc E) \to B(\mc E)$ of order $0$ which equals $\mc L_{H,s}$ for $\Rea s>\tfrac12$. Moreover, for each $f\in B(\mc E)$ and $z\in \mc E$, the function $s\mapsto \wt{\mc L}_{H,s}f(z)$ is meromorphic with poles in $P$, and the map $(s,z)\mapsto \wt{ \mc L}_{H,s}f(z)$ is continuous on $(\C\setminus P) \times \mc E$. In Theorem~\ref{final} below we prove that $\det(1-\wt{\mc L}_{H,s})$ is meromorphic.

\begin{prop}\label{meroext}
The map $s\mapsto \mc L_{H,s}$ extends to a meromorphic function on $\C$ with values in nuclear operators of order $0$. The possible poles are located at $s=(1-k)/2$, $k\in \N_0$. They are all simple. For each pole $s_0$, there is a neighborhood $U$ of $s_0$ such that the meromorphic extension $\wt{\mc L}_{H,s}$ is of the form
\[
 \wt{\mc L}_{H,s} = \frac{1}{s-s_0}\mc A_s + \mc B_s
\]
where the operators $\mc A_s$ and $\mc B_s$ are holomorphic on $U$ and $\mc A_s$ is of rank at most $4$.
\end{prop}

\begin{proof}
Applying the strategy of \cite{Mayer_thermo} to the four maps
\begin{align*}
\Psi_1 & \colon\left\{
\begin{array}{ccl}
\{\Rea s > \tfrac12\} & \to & \{\text{operators $B(\mc E_{q-1})\to B(\mc E_1)$}\}
\\
s & \mapsto & \sum_{n\in\N} \tau_s(h_{q-1}^n)
\end{array}
\right.
\\
\Psi_2 & \colon\left\{
\begin{array}{ccl}
 \{\Rea s > \tfrac12\} & \to & \{\text{operators $B(\mc E_{q-1})\to B(\mc E_r)$}\}
\\
s & \mapsto & \sum_{n\in\N} \tau_s(h_{q-1}^n)
\end{array}
\right.
\\
\Psi_3 & \colon\left\{
\begin{array}{ccl}
\{\Rea s > \tfrac12\} & \to & \{\text{operators $B(\mc E_{1})\to B(\mc E_r)$}\}
\\
s & \mapsto & \sum_{n\in\N} \tau_s(h_1^n)
\end{array}
\right.
\\
\Psi_4 & \colon\left\{
\begin{array}{ccl}
 \{\Rea s > \tfrac12\} & \to & \{\text{operators $B(\mc E_{1})\to B(\mc E_{q-1})$}\}
\\
s & \mapsto & \sum_{n\in\N} \tau_s(h_1^n)
\end{array}
\right.,
\end{align*}
the claim follows immediately.
\end{proof}

\begin{thm}\label{final}
The Fredholm determinant $\det(1-\mc L_{H,s})$ extends to a meromorphic function on $\C$ whose  possible  poles  are located at $s=(1-k)/2$, $k\in\N_0$. The order of a pole is at most $4$.
\end{thm}

\begin{proof}
Let $\wt{\mc L}_{H,s}$ be the meromorphic extension of $\mc L_{H,s}$ from Proposition~\ref{meroext} and recall the restriction map 
\[
 r= \sum_{k\in\N_0} \varphi_k \otimes e_k
\]
from the proof of Proposition~\ref{tracerepr}. Set $P\sceq \{ (1-k)/2 \mid k\in\N_0\}$. On $\C\setminus P$ we have
$\wt{\mc L}_{H,s} = \wt{\mc L}_{H,s} \circ r$. Hence
\[
 \Tr \wt{\mc L}^n_{H,s} = \sum_{k\in\N_0} \varphi_k\big(\wt{\mc L}_{H,s}e_k\big)
\]
for each $n\in\N$. The latter map is holomorphic on $\C\setminus P$, thus the Fredholm determinant $\det(1-\wt{\mc L}_{H,s})$ is so.

Let $s_0\in P$ and pick a neighborhood $U$ of $s_0$ such that the operator $\wt{\mc L}_{H,s}$ is of the form
\[
 \wt{\mc L}_{H,s} = \frac{1}{s-s_0} \mc A_s + \mc B_s
\]
with $\mc A_s$ and $\mc B_s$ holomorphic operators and $\mc A_s$ of rank at most $4$. By \cite{Grothendieck_fredholm},
\[
 \det\big(1-\wt{\mc L}_{H,s}\big) = \sum_{n=0}^\infty (-1)^n\Tr \bigwedge\nolimits^{\!\!n} \wt{\mc L}_{H,s}.
\]
Since the rank of $\mc A_s$ is at most $4$, evaluating the exterior product shows that for each $n\in\N_0$, the map
\[
 s\mapsto (s-s_0)^4\Tr \bigwedge\nolimits^{\!\!n} \wt{\mc L}_{H,s}
\]
is holomorphic on $U$. Thus,
\[
(s-s_0)^{-4}\det\big(1-\wt{\mc L}_{H,s}\big) 
\]
is holomorphic on $U$ as well. This completes the proof.
\end{proof}

\subsection{Factorization of the Fredholm determinant}\label{sec_factorization}

In this chapter we show that the transfer operator $\mc L_{H,s}$ and its Fredholm determinant split into two parts, one of which conjecturally is related to odd Maass cusp forms, the other one to even Maass cusp forms. An easy calculation shows that the transfer operator $\mc L_{H,s}$ commutes with 
\[
\mc J \sceq 
\begin{pmatrix}
 & & \tau_s(J)
\\
 & \tau_s(J) & 
\\
\tau_s(J) & &
\end{pmatrix}.
\]
Recall the definition of the sets $E_{\st, k}$, $k = 1,\ldots, q-1$, from \eqref{eq:defE}, 
and let $m \sceq \lfloor \tfrac{q+1}2 \rfloor$.
Suppose first that $q$ is even. We define
\begin{align*}
\hphantom{xxxxxxxxx} E_{r,a} & \sceq \left(-\frac{\lambda-1}{\lambda+1},0 \right] &&\supseteq \bigcup_{k=m+1}^{q-2} E_k \cup h_m^{-1}.(-1,0],\hphantom{xxxxxxxxx}
\\
E_{r,b} & \sceq  \left(0,\frac{\lambda-1}{\lambda+1}\right) && \supseteq \bigcup_{k=2}^{m-1} E_k \cup h_m^{-1}.(0,1).
\end{align*}
Representing functions \wrt $\big(\chi_{E_{q-1}}, \chi_{E_{r,a}}, \chi_{E_{r,b}}, \chi_{E_1}\big)$, the transfer operator $\mc L_{H,s}$ is represented by the matrix 
\[
 \mc L_{H,s} =
\begin{pmatrix}
0 & \sum\limits_{k=m}^{q-2} \tau_s(h_k) & \sum\limits_{k=2}^{m}\tau_s(h_k) & \sum\limits_{n\in\N}\tau_s(h_1^n)
\\
\sum\limits_{n\in\N}\tau_s(h_{q-1}^n) & \sum\limits_{k=m}^{q-2}\tau_s(h_k) & \sum\limits_{k=2}^{m}\tau_s(h_k) & \sum\limits_{n\in\N}\tau_s(h_1^n)
\\
\sum\limits_{n\in\N}\tau_s(h_{q-1}^n) & \sum\limits_{k=m}^{q-2}\tau_s(h_k) & \sum\limits_{k=2}^m\tau_s(h_k) & \sum\limits_{n\in\N}\tau_s(h_1^n)
\\
\sum\limits_{n\in\N}\tau_s(h_{q-1}^n) & \sum\limits_{k=m}^{q-2}\tau_s(h_k) & \sum\limits_{k=2}^m\tau_s(h_k) & 0 
\end{pmatrix}.
\]
If $f$ is a common eigenfunction of $\mc L_{H,s}$ and 
\[
\mc J = 
\begin{pmatrix}
& &  & \tau_s(J)
\\
&  & \tau_s(J)
\\
& \tau_s(J)
\\
\tau_s(J)
\end{pmatrix},
\]
then $f$ is of the form
\[
 f = 
\begin{pmatrix}
f_1
\\
f_2
\\
\pm\tau_s(J)f_2
\\
\pm\tau_s(J)f_1
\end{pmatrix}
\]
where one has to take the same sign in the last two entries. Then we find
\[
\begin{pmatrix}
f_1
\\
f_2
\\
\pm\tau_s(J)f_2
\\
\pm\tau_s(J)f_1
\end{pmatrix}
=
\begin{pmatrix}
\sum\limits_{k=m}^{q-2} \tau_s(h_k)f_2 \pm \sum\limits_{k=2}^{m}\tau_s(h_kJ)f_2 \pm \sum\limits_{n\in\N}\tau_s(h_1^nJ)f_1
\\
\sum\limits_{n\in\N}\tau_s(h_{q-1}^n)f_1 + \sum\limits_{k=m}^{q-2}\tau_s(h_k)f_2 \pm \sum\limits_{k=2}^{m}\tau_s(h_kJ)f_2 \pm \sum\limits_{n\in\N}\tau_s(h_1^nJ)f_1
\\
\sum\limits_{n\in\N}\tau_s(h_{q-1}^n)f_1 + \sum\limits_{k=m}^{q-2}\tau_s(h_k)f_2 \pm \sum\limits_{k=2}^m\tau_s(h_kJ)f_2 \pm \sum\limits_{n\in\N}\tau_s(h_1^nJ)f_1
\\
\sum\limits_{n\in\N}\tau_s(h_{q-1}^n)f_1 + \sum\limits_{k=m}^{q-2}\tau_s(h_k)f_2 \pm\sum\limits_{k=2}^m\tau_s(h_kJ)f_2
\end{pmatrix}.
\]
We define
\[
 \mc L_{H,s}^{\pm} \sceq 
\begin{pmatrix}
\pm \sum\limits_{n\in\N}\tau_s(h_1^nJ) & \sum\limits_{k=m}^{q-2}\tau_s(h_k) \pm \sum\limits_{k=2}^{m}\tau_s(h_kJ)
\\
\sum\limits_{n\in\N}\tau_s(h_{q-1}^n) \pm \sum\limits_{n\in\N}\tau_s(h_1^nJ) & \sum\limits_{k=m}^{q-2}\tau_s(h_k) \pm \sum\limits_{k=2}^{m}\tau_s(h_kJ)
\end{pmatrix}.
\]

Suppose now that $q$ is odd. Then we set
\begin{align*}
\hphantom{xxxxxxxxxxxxxx} E_{r,a} & \sceq \left(-\frac{\lambda-1}{\lambda+1},0\right) && \supseteq \bigcup_{k=m}^{q-2} E_{\st,k}, \hphantom{xxxxxxxxxxxxxx}  
\\
E_{r,b} & \sceq \left(0,\frac{\lambda-1}{\lambda+1}\right) &&  \supseteq \bigcup_{k=2}^{m-1} E_{\st,k}.
\end{align*}
Representing function \wrt $\big(\chi_{E_{q-1}},\chi_{E_{r,a}},\chi_{E_{r,b}},\chi_{E_1}\big)$, the transfer operator $\mc L_{H,s}$ is represented by 
\[
\mc L_{H,s} =
\begin{pmatrix}
0 & \sum\limits_{k=m}^{q-2}\tau_s(h_k) & \sum\limits_{k=2}^{m-1}\tau_s(h_k) & \sum\limits_{n\in\N}\tau_s(h_1^n)
\\
\sum\limits_{n\in\N}\tau_s(h_{q-1}^n) & \sum\limits_{k=m}^{q-2}\tau_s(h_k) & \sum\limits_{k=2}^{m-1}\tau_s(h_k) & \sum\limits_{n\in\N}\tau_s(h_1^n)
\\
\sum\limits_{n\in\N}\tau_s(h_{q-1}^n) & \sum\limits_{k=m}^{q-2}\tau_s(h_k) & \sum\limits_{k=2}^{m-1}\tau_s(h_k) & \sum\limits_{n\in\N}\tau_s(h_1^n)
\\
\sum\limits_{n\in\N}\tau_s(h_{q-1}^n) & \sum\limits_{k=m}^{q-2}\tau_s(h_k) & \sum\limits_{k=2}^{m-1}\tau_s(h_k) & 0
\end{pmatrix}.
\]
Moreover, as above,
\[
\begin{pmatrix}
f_1
\\
f_2
\\
\pm\tau_s(J)f_2
\\
\pm\tau_s(J)f_1
\end{pmatrix}
=
\begin{pmatrix}
 \sum\limits_{k=m}^{q-2}\tau_s(h_k)f_2 \pm \sum\limits_{k=2}^{m-1}\tau_s(h_kJ)f_2 \pm \sum\limits_{n\in\N}\tau_s(h_1^nJ)f_1
\\
\sum\limits_{n\in\N}\tau_s(h_{q-1}^n)f_1 + \sum\limits_{k=m}^{q-2}\tau_s(h_k)f_2 \pm \sum\limits_{k=2}^{m-1}\tau_s(h_kJ)f_2 \pm \sum\limits_{n\in\N}\tau_s(h_1^nJ)f_1
\\
\sum\limits_{n\in\N}\tau_s(h_{q-1}^n)f_1 + \sum\limits_{k=m}^{q-2}\tau_s(h_k)f_2 \pm \sum\limits_{k=2}^{m-1}\tau_s(h_kJ)f_2 \pm \sum\limits_{n\in\N}\tau_s(h_1^nJ)f_1
\\
\sum\limits_{n\in\N}\tau_s(h_{q-1}^n)f_1 + \sum\limits_{k=m}^{q-2}\tau_s(h_k)f_2 \pm \sum\limits_{k=2}^{m-1}\tau_s(h_kJ)f_2
\end{pmatrix}.
\]
We define
\[
 \mc L_{H,s}^{\pm} \sceq
\begin{pmatrix}
\pm\sum\limits_{n\in\N} \tau_s(h_1^nJ) & \sum\limits_{k=m}^{q-2}\tau_s(h_k) \pm \sum\limits_{k=2}^{m-1}\tau_s(h_kJ)
\\
\sum\limits_{n\in\N}\tau_s(h_{q-1}^n) \pm \sum\limits_{n\in\N}\tau_s(h_1^nJ) & \sum\limits_{k=m}^{q-2}\tau_s(h_k) \pm \sum\limits_{k=2}^{m-1}\tau_s(h_kJ) 
\end{pmatrix}.
\]

\begin{prop}\label{twistnuclear} For $\Rea s > \tfrac12$, the operators $\mc L_{H,s}^\pm$ are nuclear of order $0$ on $B(\mc E_{q-1})\times B(\mc E_r)$ and we have
\[
 \det(1-\mc L_{H,s}) = \det(1-\mc L_{H,s}^+) \cdot \det(1-\mc L_{H,s}^-).
\]
Moreover, the Fredholm determinants $\det(1-\mc L_{H,s}^+)$ and $\det(1-\mc L_{H,s}^-)$  extend to meromorphic functions on $\C$ with possible poles at $s=(1-k)/2$, $k\in\N_0$. 
\end{prop}

\begin{proof}
Let 
\[
 \mc P \sceq \frac{1}{\sqrt{2}}
\begin{pmatrix}
1 & & & \tau_s(J)
\\
 & 1 & \tau_s(J) & 
\\
 & \tau_s(J) & -1 & 
\\
\tau_s(J) & & & -1
\end{pmatrix}
\]
and
\[
 \mc R \sceq \mat{}{\tau_s(J)}{\tau_s(J)}{}.
\]
Then $\mc P$ and $\mc R$ are self-inverse and
\[
 \mc P \mc L_{H,s} \mc P = \mat{\mc L_{H,s}^+}{}{}{\mc R \mc L_{H,s}^- \mc R}.
\]
Hence
\begin{align*}
\det(1-\mc L_{H,s}) & = \det(1 - \mc P \mc L_{H,s} \mc P) = \det(1 - \mc L_{H,s}^+)\cdot \det(1- \mc R \mc L_{H,s}^- \mc R)
\\
& = \det(1-\mc L_{H,s}^+) \cdot \det(1 - \mc L_{H,s}^-).
\end{align*}
The remaining statements now follow from Proposition~\ref{isnuclear} and Theorem~\ref{final}.
\end{proof}
\par

\subsection{Comparison of the slow and the fast system}\label{sec_conj}

Because of the very geometric construction (see \cite{Pohl_Symdyn2d} and the discussion below) of the slow and fast discrete dynamical system, we expect that the following comparison statement holds.

\begin{conj}\label{conj}
Let $\Rea s = \tfrac12$. The space of $1$-eigenfunctions of the meromorphic extension of $\mc L_{H,s}^+$ is linear isomorphic to $\FE_s(\R^+)^{\dec,+}_\omega$ and hence corresponds to even Maass cusp forms. The space of $1$-eigenfunctions of the meromorphic extension of $\mc L_{H,s}^-$ is linear isomorphic to $\FE_s(\R^-)^{\dec,-}_\omega$ and hence corresponds to odd Maass cusp forms.
\end{conj}

A couple of arguments support this conjecture: The functional equation in the definition of the space $\FE_s(\R^+)^{\dec}_\omega$ is the defining equation for the $1$-eigenfunctions of the transfer operator with parameter $s$ arising from a discretization of the geodesic flow on $G_q\backslash \h$. The transfer operator $\mc L_{H,s}$ is also associated to a discretization of this geodesic flow. The latter discretization is strongly related to the first discretization, it is a specific acceleration of the first discretization. Heuristically, both discretizations and both transfer operator families should contain the same geometric information. For $\Rea s = \tfrac12$, the Selberg trace formula implies (via Theorem~\ref{final}) that $\mc L_{H,s}$ has a $1$-eigenfunction (more precisely, its meromorphic continuation $\wt{\mc L}_{H,s}$ has a $1$-eigenfunction) if and only if $\FE_s(\R^+)^{\dec}_\omega$ is non-trivial. 

For the splitting of the space $\FE_s(\R^+)^{\dec}_\omega$ into the subspaces $\FE_s(\R^+)^{\dec,\pm}_\omega$ we used the symmetry $\tau_s(Q)$. This splitting is the same as building two transfer operators $\mc L_{F,s}^\pm$ from the slow transfer operator $\mc L_{F,s}$ along the lines of the construction above. The $1$-eigenfunctions of $\mc L_{F,s}^\pm$ are then characterized by the functional equation in $\FE_s(\R^+)^{\dec,\pm}_\omega$. The same symmetry (after the necessary conjugation with $\mc T$) is used to decompose $\mc L_{H,s}$ into the operators $\mc L_{H,s}^\pm$. Invoking the idea that the essential geometric properties are captured by both discretizations, the $1$-eigenfunctions of $\mc L_{H,s}^\pm$ (more precisely, of $\wt{\mc L}_{H,s}^\pm$) should correspond to the $1$-eigenfunctions of $\mc L_{F,s}^\pm$, resp., of appropriate regularity.

For the case $q=3$, the transfer operators $\mc L_{H,s}^\pm$ are $\pm \sum_{n\in\N} \tau_s(h_1^nJ)$, which is  Mayer's transfer operator \cite{Mayer_thermoPSL} (up to the conjugation by $\mc T$). In this case, Conjecture~\ref{conj} has been established by Efrat \cite{Efrat_spectral} on the level of parameters $s$ for the transfer operators and eigenvalues $s(1-s)$ for Maass cusp forms, and by Chang and Mayer \cite{Chang_Mayer_transop} as well as by Lewis and Zagier \cite{Lewis_Zagier} on the level of eigenfunctions and Maass cusp forms.

To end, we prove two lemmas which might be helpful in proving Conjecture~\ref{conj}.

\begin{lemma}
If $f=(f_{q-1},f_r)$ is a $1$-eigenfunction of $\mc L_{H,s}^+$ or $\mc L_{H,s}^-$, then $f_r$ is determined by $f_{q-1}$.
\end{lemma}

\begin{proof}
W.l.o.g.\@ suppose that $q$ is odd and $f$ is a $1$-eigenfunction of $\mc L_{H,s}$. Then 
\begin{align*}
f_{q-1} & = \sum_{n\in\N} \tau_s(h_1^nJ)f_{q-1} + \sum_{k=m}^{q-2}\tau_s(h_k)f_r + \sum_{k=2}^{m-1}\tau_s(h_kJ)f_r
\\
f_r & = \sum_{n\in\N} \tau_s(h_{q-1}^n) f_{q-1} + \sum_{n\in\N} \tau_s(h_1^nJ)f_{q-1} + \sum_{k=m}^{q-2}\tau_s(h_k)f_r + \sum_{k=2}^{m-1}\tau_s(h_kJ)f_r.
\end{align*}
On $\mc E_{q-1}\cap \mc E_r$, which is an open non-empty set, we have
\[
 f_r = \sum_{n\in\N_0} \tau_s(h_{q-1}^n) f_{q-1}.
\]
Since $f_r$ is holomorphic, it is determined by its values on $\mc E_{q-1}\cap \mc E_r$.
\end{proof}

The solutions of the functional equations \eqref{funceqmod1}-\eqref{funceqmod4} are determined uniquely by their values on a ``fundamental
domain'', as given in the following lemma. As in the corresponding discussion in \cite{Lewis_Zagier}
it is not clear how to recognize the functions on the fundamental domain that extend
to functions on $\mathbb R^+$ with sufficient regularity to lie in $\FE_s(\R^+)^{\dec,\pm}_\omega$.
\par
\begin{lemma}
Let $\psi$ be a solution to one of the functional equations \eqref{funceqmod1}-\eqref{funceqmod4}. Then $\psi$ is determined by its values on the interval $[1,1+\lambda]$.
\end{lemma}

\begin{proof}
We suppose that $q$ is odd and $\psi$ satisfies \eqref{funceqmod3} (all the other cases are analogous). Then $\psi=\tau_s(Q)\psi$. Hence it suffices to show that $\psi$ is determined on $[1,\infty)$. The functional equation \eqref{funceqmod3} applied to $x-\lambda$ gives
\begin{align}\label{iteration}
\psi(x)  & = \psi(x-\lambda) - j_s(g_1^{-1}Q, x-\lambda) \psi\left( \frac{1}{x-\lambda} + \lambda\right) 
\\
& \qquad\qquad - \sum_{k=2}^{m-1} \tau_s(Qg_k)\psi(x-\lambda) - \sum_{k=2}^{m-1}\tau_s(g_k)\psi(x-\lambda).\nonumber
\end{align}
Now
\[
 [1,1+\lambda] \supseteq \bigcup_{k=2}^{m-1}g_k^{-1}.\R^+
\]
and
\[
 \frac{1}{x-\lambda} + \lambda \in [1,1+n\lambda]
\]
for $x\in [1+n\lambda, 1+(n+1)\lambda]$ and $n\in\N$. Thus, iteratively applying \eqref{iteration} to the intervals $[1+\lambda, 1+2\lambda]$, $[1+2\lambda, 1+3\lambda]$, $\ldots$, determines $\psi$ on $[1,\infty)$.
\end{proof}


\begin{thebibliography}{HMM05}

\bibitem[BLZ]{BLZ_part2}
R.~Bruggeman, J.~Lewis, and D.~Zagier, \emph{Period functions for {M}aass wave
  forms. {II}: cohomology}, preprint.

\bibitem[BM09]{Bruggeman_Muehlenbruch}
R.~W. Bruggeman and T.~M{\"u}hlenbruch, \emph{Eigenfunctions of transfer
  operators and cohomology}, J. Number Theory \textbf{129} (2009), no.~1,
  158--181.

\bibitem[Bru97]{Bruggeman}
R.~Bruggeman, \emph{Automorphic forms, hyperfunction cohomology, and period
  functions}, J. reine angew. Math. \textbf{492} (1997), 1--39.

\bibitem[CM99]{Chang_Mayer_transop}
C.-H. Chang and D.~Mayer, \emph{The transfer operator approach to {S}elberg's
  zeta function and modular and {M}aass wave forms for {${\rm PSL}(2,{\bf
  Z})$}}, Emerging applications of number theory ({M}inneapolis, {MN}, 1996),
  IMA Vol. Math. Appl., vol. 109, Springer, New York, 1999, pp.~73--141.

\bibitem[CM01a]{Chang_Mayer_eigen}
\bysame, \emph{Eigenfunctions of the transfer operators and the period
  functions for modular groups}, Dynamical, spectral, and arithmetic zeta
  functions ({S}an {A}ntonio, {TX}, 1999), Contemp. Math., vol. 290, Amer.
  Math. Soc., Providence, RI, 2001, pp.~1--40.

\bibitem[CM01b]{Chang_Mayer_extension}
\bysame, \emph{An extension of the thermodynamic formalism approach to
  {S}elberg's zeta function for general modular groups}, Ergodic theory,
  analysis, and efficient simulation of dynamical systems, Springer, Berlin,
  2001, pp.~523--562.

\bibitem[DH07]{Deitmar_Hilgert}
A.~Deitmar and J.~Hilgert, \emph{A {L}ewis correspondence for submodular
  groups}, Forum Math. \textbf{19} (2007), no.~6, 1075--1099.

\bibitem[DIPS85]{DIPS}
J.-M. Deshouillers, H.~Iwaniec, R.~S. Phillips, and P.~Sarnak, \emph{Maass cusp
  forms}, Proc. Nat. Acad. Sci. U.S.A. \textbf{82} (1985), no.~11, 3533--3534.

\bibitem[Efr93]{Efrat_spectral}
I.~Efrat, \emph{Dynamics of the continued fraction map and the spectral theory
  of {${\rm SL}(2,\bold Z)$}}, Invent. Math. \textbf{114} (1993), no.~1,
  207--218. \MR{1235024 (94h:11052)}

\bibitem[Fis87]{Fischer}
J.~Fischer, \emph{An approach to the {S}elberg trace formula via the {S}elberg
  zeta-function}, Lecture Notes in Mathematics, vol. 1253, Springer-Verlag,
  Berlin, 1987.

\bibitem[FMM07]{Fraczek_Mayer_Muehlenbruch}
M.~Fraczek, D.~Mayer, and T.~M{\"u}hlenbruch, \emph{A realization of the
  {H}ecke algebra on the space of period functions for {$\Gamma_0(n)$}}, J.
  Reine Angew. Math. \textbf{603} (2007), 133--163.

\bibitem[Fri96]{Fried_triangle}
D.~Fried, \emph{Symbolic dynamics for triangle groups}, Invent. Math.
  \textbf{125} (1996), no.~3, 487--521.

\bibitem[Gro55]{Grothendieck_produit}
A.~Grothendieck, \emph{Produits tensoriels topologiques et espaces
  nucl\'eaires}, Mem. Amer. Math. Soc. \textbf{1955} (1955), no.~16, 140.

\bibitem[Gro56]{Grothendieck_fredholm}
\bysame, \emph{La th\'eorie de {F}redholm}, Bull. Soc. Math. France \textbf{84}
  (1956), 319--384.

\bibitem[Hej83]{Hejhal_stf2}
D.~A. Hejhal, \emph{The {S}elberg trace formula for {${\rm PSL}(2,\,{\bf R})$}.
  {V}ol. 2}, Lecture Notes in Mathematics, vol. 1001, Springer-Verlag, Berlin,
  1983.

\bibitem[Hej92]{Hejhal_eigenvalueshecke}
\bysame, \emph{Eigenvalues of the {L}aplacian for {H}ecke triangle groups},
  Mem. Amer. Math. Soc. \textbf{97} (1992), no.~469, vi+165.

\bibitem[HMM05]{Hilgert_Mayer_Movasati}
J.~Hilgert, D.~Mayer, and H.~Movasati, \emph{Transfer operators for
  {$\Gamma_0(n)$} and the {H}ecke operators for the period functions of {${\rm
  PSL}(2,{\Bbb Z})$}}, Math. Proc. Cambridge Philos. Soc. \textbf{139} (2005),
  no.~1, 81--116.

\bibitem[Lew97]{Lewis}
J.~Lewis, \emph{Spaces of holomorphic functions equivalent to the even {M}aass
  cusp forms}, Invent. Math. \textbf{127} (1997), 271--306.

\bibitem[LZ01]{Lewis_Zagier}
J.~Lewis and D.~Zagier, \emph{Period functions for {M}aass wave forms. {I}},
  Ann. of Math. (2) \textbf{153} (2001), no.~1, 191--258.

\bibitem[May76]{Mayer_zeta}
D.~Mayer, \emph{On a {$\zeta $} function related to the continued fraction
  transformation}, Bull. Soc. Math. France \textbf{104} (1976), no.~2,
  195--203.

\bibitem[May90]{Mayer_thermo}
\bysame, \emph{On the thermodynamic formalism for the {G}auss map}, Comm. Math.
  Phys. \textbf{130} (1990), no.~2, 311--333.

\bibitem[May91]{Mayer_thermoPSL}
\bysame, \emph{The thermodynamic formalism approach to {S}elberg's zeta
  function for {${\rm PSL}(2,{\bf Z})$}}, Bull. Amer. Math. Soc. (N.S.)
  \textbf{25} (1991), no.~1, 55--60.

\bibitem[MMS10]{Mayer_Muehlenbruch_Stroemberg}
D.~Mayer, T.~M{\"u}hlenbruch, and F.~Str{\"o}mberg, \emph{The transfer operator
  for the {H}ecke triangle groups}, arXiv:0912.2236, 2010.

\bibitem[Mor97]{Morita_transfer}
T.~Morita, \emph{Markov systems and transfer operators associated with cofinite
  {F}uchsian groups}, Ergodic Theory Dynam. Systems \textbf{17} (1997), no.~5,
  1147--1181.

\bibitem[MS08]{Mayer_Stroemberg}
D.~Mayer and F.~Str{\"o}mberg, \emph{Symbolic dynamics for the geodesic flow on
  {H}ecke surfaces}, J. Mod. Dyn. \textbf{2} (2008), no.~4, 581--627.

\bibitem[Poh10]{Pohl_Symdyn2d}
A.~Pohl, \emph{Symbolic dynamics for the geodesic flow on two-dimensional
  hyperbolic good orbifolds}, 2010, arXiv:1008.0367v1.

\bibitem[Pol91]{Pollicott}
M.~Pollicott, \emph{Some applications of thermodynamic formalism to manifolds
  with constant negative curvature}, Adv. in Math. \textbf{85} (1991),
  161--192.

\bibitem[PS85a]{Phillips_Sarnak_cuspforms}
R.~S. Phillips and P.~Sarnak, \emph{On cusp forms for co-finite subgroups of
  {${\rm PSL}(2,{\bf R})$}}, Invent. Math. \textbf{80} (1985), no.~2, 339--364.

\bibitem[PS85b]{Phillips_Sarnak_weyl}
\bysame, \emph{The {W}eyl theorem and the deformation of discrete groups},
  Comm. Pure Appl. Math. \textbf{38} (1985), no.~6, 853--866.

\bibitem[Rue76]{Ruelle_zeta}
D.~Ruelle, \emph{Zeta-functions for expanding maps and {A}nosov flows}, Invent.
  Math. \textbf{34} (1976), no.~3, 231--242.

\bibitem[Rue94]{Ruelle_dynzeta}
D.~Ruelle, \emph{Dynamical zeta functions for piecewise monotone maps of the
  interval}, CRM Monograph Series, vol.~4, American Mathematical Society,
  Providence, RI, 1994.

\bibitem[Sel56]{Selberg}
A.~Selberg, \emph{Harmonic analysis and discontinuous groups in weakly
  symmetric {R}iemannian spaces with applications to {D}irichlet series}, J.
  Indian Math. Soc. (N.S.) \textbf{20} (1956), 47--87.

\bibitem[Ser85]{Series}
C.~Series, \emph{The modular surface and continued fractions}, J. London Math.
  Soc. (2) \textbf{31} (1985), no.~1, 69--80.

\end{thebibliography}

\providecommand{\bysame}{\leavevmode\hbox to3em{\hrulefill}\thinspace}
\providecommand{\MR}{\relax\ifhmode\unskip\space\fi MR }
\providecommand{\MRhref}[2]{%
  \href{http://www.ams.org/mathscinet-getitem?mr=#1}{#2}
}
\providecommand{\href}[2]{#2}

\end{document}